\RequirePackage[english]{babel}
\documentclass[a4paper,twoside,11pt]{amsart}
\usepackage[latin1]{inputenc}
\usepackage{amsmath}
\usepackage{amssymb}
\usepackage{amsfonts}

\usepackage{enumerate}
\usepackage[matrix,arrow,cmtip]{xy} % all
\SelectTips{cm}{}
\newdir{(}{{}*!/-5pt/@^{(}}
\entrymodifiers={+!!<0pt,\fontdimen22\textfont2>}

\usepackage[nosubsections]{mytheorems2}

\theoremstyle{dremark}
\newtheorem*{acknowledgement}{Acknowledgement}

\newcommand{\pref}[1]{\eqref{#1}}
\newcommand{\enumref}[1]{\ref{#1}}

\newcommand\inj{\hookrightarrow}
\newcommand\surj{\twoheadrightarrow}
\newcommand\iso{\cong} % Isomorphism (perhaps \simeq is better?)
\def\map#1#2#3{#1\,:\, #2\rightarrow #3}
\def\surjmap#1#2#3{#1\,:\, #2\twoheadrightarrow #3}
\newcommand{\id}[1]{\mathrm{id}_{#1}} % Identity map

\newcommand{\N}{\mathbb{N}}
\newcommand{\Z}{\mathbb{Z}}
\newcommand{\Q}{\mathbb{Q}}
\newcommand{\F}{\mathbb{F}}

\newcommand{\I}{\mathcal{I}}

\newcommand{\Sets}{\mathbf{Sets}}
\newcommand{\Mod}{\mathbf{Mod}} % Mod
\newcommand{\Alg}{\mathbf{Alg}} % Alg
\newcommand{\xAlg}{\text{\nobreakdash--}\Alg} % A-Alg
\newcommand{\xMod}{\text{\nobreakdash--}\Mod} % A-Mod

\newcommand{\fF}{\mathcal{F}}

\newcommand{\divides}{\mid}
\newcommand{\notdivide}{\nmid}

\newcommand{\ceil}[1]{\left\lceil#1\right\rceil}
\newcommand{\floor}[1]{\left\lfloor#1\right\rfloor}

\newcommand{\SG}[1]{{\mathfrak{S}_{#1}}} % Symmetric group
\DeclareMathOperator{\Hom}{Hom}
\DeclareMathOperator{\Sym}{Sym}
\DeclareMathOperator{\image}{im}
\newcommand{\ord}{\mathrm{ord}}

\newcommand{\Spec}{\mathrm{Spec}}
\newcommand{\red}{\mathrm{red}}
\newcommand{\Chow}{\mathrm{Chow}}
\newcommand{\A}[1]{\mathbb{A}^{#1}}    % Affine space
\newcommand{\PR}[1]{\mathbb{P}^{#1}}    % Projective space

\newcommand{\T}{\mathrm{T}}
\newcommand{\TS}{\mathrm{TS}}
\newcommand{\SYM}{\mathrm{S}}    % Symmetric algebra

\newcommand{\Hompl}[1]{{\mathrm{Pol}^{#1}}}
\newcommand{\Hommpl}[1]{{\mathrm{Pol}_{\mathrm{mult}}^{#1}}}

\def\bigtimes{\mathop{\times}}
\DeclareMathOperator{\mdeg}{mdeg}
\def\basis#1{\mathbf{e}_{#1}}

\newcommand{\mon}{\mathcal{M}}
\newcommand{\monpos}{{\mon^{*}}}
\newcommand{\calC}{\mathcal{C}}
\newcommand{\calS}{\mathcal{S}}
\newcommand{\calT}{\mathcal{T}}

\newcommand{\primes}{\mathcal{P}}

\newcommand{\gammaid}[1]{\mathbf{1}^{#1}}
\newcommand{\ltdeg}[1]{{\left\langle#1\right\rangle}}
\newcommand{\msb}[1]{\mathbf{z}_{#1}}

\newcommand{\Qpol}[2]{Q_{#1}(#2)}
\newcommand{\Qpolp}[2]{Q_{#1}\bigl(#2\bigr)}

\newcommand{\Qp}[1]{\Qpol{p}{#1}}
\newcommand{\Qpp}[1]{\Qpolp{p}{#1}}

\begin{document}

\title[The ring of multisymmetric functions]
{A minimal set of generators for the\\ring of multisymmetric functions}
\author[D. Rydh]{David Rydh}
\address{Department of Mathematics, KTH, 100 44 Stockholm, Sweden}
\email{dary@math.kth.se}
\date{2007-05-21}
%\subjclass{Primary 13A50, 05E05;
%  Secondary 14L30, 14C05}

%% \begin{abstract}
%% The purpose of this article is to give, for any (commutative)
%% ring $A$, an explicit minimal set of generators for the ring of
%% multisymmetric functions $\TS^d_A(A[x_1,\dots,x_r])=
%% \bigl(A[x_1,\dots,x_r]^{\otimes_A d}\bigr)^{\SG{d}}$ as an $A$-algebra.
%% In characteristic zero, i.e. when $A$ is a $\Q$-algebra, a minimal set of
%% generators has been known since the 19\textsuperscript{th}
%% century~\cite{schlafli,junker_93}. A rather small generating set in the
%% general case has also recently been given by
%% Vaccarino~\cite[Thm. 1]{vaccarino_gen_TS} but it is not minimal in general.
%% We also give a sharp degree bound on the generators, improving the degree
%% bound previously obtained by Fleischmann~\cite[Thm. 4.6]{fleischmann}.

%% As $\Gamma^d_A(A[x_1,\dots,x_r])=\TS^d_A(A[x_1,\dots,x_r])$ we also obtain
%% generators for divided powers algebras; If $B$ is a finitely generated
%% $A$-algebra with a given surjection $A[x_1,x_2,\dots,x_r]\to B$ then using
%% the corresponding surjection $\Gamma^d_A(A[x_1,\dots,x_r])\to \Gamma^d_A(B)$
%% we get generators for $\Gamma^d_A(B)$.
%% \end{abstract}

\begin{abstract}
The purpose of this article is to give, for any (commutative) ring $A$, an
explicit minimal set of generators for the ring of multisymmetric functions
$\TS^d_A(A[x_1,\dots,x_r])= \bigl(A[x_1,\dots,x_r]^{\otimes_A
d}\bigr)^{\SG{d}}$ as an $A$-algebra. In characteristic zero, i.e. when $A$ is
a $\Q$-algebra, a minimal set of generators has been known since the
19\textsuperscript{th} century. A rather small generating set in the general
case has also recently been given by Vaccarino but it is not minimal in
general.  We also give a sharp degree bound on the generators, improving the
degree bound previously obtained by Fleischmann.

As $\Gamma^d_A(A[x_1,\dots,x_r])=\TS^d_A(A[x_1,\dots,x_r])$ we also obtain
generators for divided powers algebras: If $B$ is a finitely generated
$A$-algebra with a given surjection $A[x_1,x_2,\dots,x_r]\to B$ then using
the corresponding surjection $\Gamma^d_A(A[x_1,\dots,x_r])\to \Gamma^d_A(B)$
we get generators for $\Gamma^d_A(B)$.
\end{abstract}

\maketitle

%%%%%%%%%%%%%%%%%%%%%%%%%%%%%%%%%%%%%%%%%%%%%%%%%

\begin{section}{Introduction}

Let $k$ be a field of characteristic zero. Explicit generators for the ring of
multisymmetric functions $\bigl(k[x_1,\dots,x_r]^{\otimes_k d}\bigr)^{\SG{d}}$
have been known since the nineteenth century, cf.~\cite{schlafli,junker_93}.
At the end of the same century non-constructive methods began to appear, in
particular Hilbert's basis theorem~\cite{hilbert_basis_theorem}. An easy
consequence of this theorem is that if a finite group $G$ acts linearly on a
polynomial ring over a field $k$ and $|G|$ is invertible in $k$, then the
invariant ring is finitely generated, cf.~\cite[\S 57]{weber_lehrbuch_vol2}.
In particular, we may deduce that
$\bigl(k[x_1,\dots,x_r]^{\otimes_k d}\bigr)^{\SG{d}}$ is finitely generated
without finding explicit generators.

% \S 42 in First edition of Weber's second volume of ``Lehrbuch der Algebra''

The first result on the finiteness of the invariant ring of a group action in
characteristic $p$ was given by Noether~\cite{noether_endlichkeitssatz_char_p}.
Her argument is essentially the following: Let $A$ be a noetherian ring
and $G=\{g_1,g_2,\dots,g_m\}$ a finite group acting on $B=A[x_1,\dots,x_n]$.
Let $C=A[e_{11},e_{12},\dots,e_{nm}]\subseteq B$ where $e_{ij}$ is the
$j$\textsuperscript{th} elementary symmetric function
in $g_1(x_i),g_2(x_i),\dots,g_m(x_i)$. The $A$-algebra $C$ is finitely
generated and hence noetherian. As $B$ is finite over $C$ and
$C\subseteq B^G\subseteq B$ it follows that $B^G$ is finite over $C$ and
thus finitely generated as an $A$-algebra.
In particular $\bigl(A[x_1,\dots,x_r]^{\otimes_A d}\bigr)^{\SG{d}}$ is
finitely generated for any noetherian ring $A$.

The abstract methods partly removed the need for explicit generators.
However, interest in effective answers reappeared in the end of the twentieth
century. One of the first results in this direction was given by Campbell,
Hughes and Pollack~\cite{chp_gen_TS} who showed that the ring of multisymmetric
functions can be generated by elements of degree
$\leq\max\bigl(d,rd(d-1)/2\bigr)$. In characteristic $0$, the explicit
generators have degree $\leq d$.

Some years later Fleischmann~\cite[Thm. 4.6]{fleischmann} improved this
degree bound to $\leq\max\bigl(d,r(d-1)\bigr)$ and also
showed~\cite[Thm. 4.7]{fleischmann} that this was the best possible if $A=\F_p$
and $d=p^s$, cf. Corollary~\pref{C:non-sharp_total_degree_bound}.
Vaccarino~\cite[Thm. 1]{vaccarino_gen_TS} then used this result
to give explicit generators, cf. Theorem~\pref{T:Vaccarino_result}.
Fleischmann's degree bound is, however, not always sharp and the corresponding
generating set is not minimal.

For any positive integer $n$ and prime $p$ we let
$\Qp{n}(t)=a_s t^s+a_{s-1} t^{s-1}+\dots+a_0\in\N[t]$ where
$a_s a_{s-1}\dots a_0$ is the representation of $n$ in base $p$.
We will prove the following theorem, cf. Theorem~\pref{T:Gamma_generators}:

\begin{theorem*}
Let $A$ be any ring and $r,d\geq 1$ positive integers. The ring of
multisymmetric functions $\bigl(A[x_1,\dots,x_r]^{\otimes_A d}\bigr)^{\SG{d}}$
is \emph{minimally} generated as an $A$-algebra by the elements
$$e_k(x^\alpha)=(x^\alpha)^{\otimes k}\otimes 1^{\otimes d-k}+\dots+
1^{\otimes d-k}\otimes (x^\alpha)^{\otimes k}$$
where $e_k$ is the $k$\textsuperscript{th} elementary symmetric function on $d$
variables and $(k,\alpha)\in\{1,2,\dots,d\}\times (\N^r\setminus 0)$ are such
that $\gcd(\alpha)=1$ and either $k|\alpha|\leq d$ or there is a prime $p$, not
invertible in $A$, such that
$\Qp{k\alpha_1}+\Qp{k\alpha_2}+\dots+\Qp{k\alpha_r}\leq\Qp{d}$.
\end{theorem*}

It is then easy to obtain the following sharp degree bound,
cf. Corollary~\pref{C:sharp_total_degree_bound}:

\begin{corollary*}
Let $A$ be any ring and $r,d\geq 1$ positive integers. For any prime $p$
we let $a_p$ and $b_p$ be defined by $\Qp{d}(t)=a_p t^{b_p}+\dots$. The
ring of multisymmetric functions is then generated by the elements of
degree at most
$$\max\left\{d,\max_p\left( (a_p+r-1)p^{b_p}-r\right)\right\}$$
where the maximum is taken over every prime $p$ not invertible in $A$. Further,
every generating set contains an element attaining this bound.
\end{corollary*}

\vspace{3 mm}
The ring of multisymmetric functions is usually described as the symmetric
tensors of the polynomial ring in $r$ variables
$$\TS^d_A(A[x_1,\dots,x_r]):=
\bigl(A[x_1,\dots,x_r]^{\otimes_A d}\bigr)^{\SG{d}}.$$
Another, functorially more well-behaved, description of the multisymmetric
functions is given by the ring of divided powers $\Gamma^d_A(A[x_1,\dots,x_r])$
which is canonically isomorphic with the ring of multisymmetric functions.
For any $A$-algebra $B$ there is a canonical homomorphism
$\Gamma^d_A(B)\to \TS^d_A(B)$ which is an
isomorphism when $A$ is of characteristic $0$ or $B$ is a free $A$-module
but not in general~\cite{lundkvist_counter-ex}.
If $B\to C$ is surjective then $\Gamma^d_A(B)\to \Gamma^d_A(C)$ is surjective
but $\TS^d_A(B)\to \TS^d_A(C)$ need not be~\cite{lundkvist_counter-ex}. Thus, a
set of generators for $\Gamma^d_A(A[x_1,\dots,x_r])=\TS^d_A(A[x_1,\dots,x_r])$
will give a set of generators for $\Gamma^d_A(B)$ for any finitely generated
$A$-algebra $B$ but not for $\TS^d_A(B)$ in general. These issues are also
discussed in~\cite{rydh_gammasymchow_inprep}.

From yet another slightly different viewpoint we can describe the
multisymmetric functions as follows: Let $V$ be a vector space over a field
$k$ and $G$ a group acting linearly on $V$. Then $G$ also acts on the dual
space $V^*$ and on the functions on $V$, i.e. the ring $k[V]=\SYM(V^*)$. The
invariants of $G$ are the invariant elements of $k[V]$, i.e. the subring
$k[V]^G$. The set of \emph{vector invariants}~\cite{weyl_invariants} of $G$ is
the invariant subring $k[V^r]^G\subseteq k[V^r]$ where $G$ acts on
$V^r=V\oplus V\oplus\dots\oplus V$ by
$\sigma(v_1,v_2,\dots,v_r)=(\sigma v_1,\sigma v_2,\dots,\sigma v_r)$. The
symmetric functions are the invariants of the symmetric group on $d$ letters
$\SG{d}$ acting by permutations on $V=k^d$. The multisymmetric functions are
the vector invariants of the same action.

Closely related to the question of generators of the ring of multisymmetric
functions is the question of which relations these generators satisfy. In
characteristic 0 the relations between the generators were thoroughly studied
by Junker~\cite{junker_91,junker_93,junker_94} in the nineteenth century. More
recently, Vaccarino gave in~\cite[Thm. 2]{vaccarino_gen_TS} relations for
his set of generators mentioned above. In this article however, we will not
discuss the relations that the generators satisfy.

\vspace{3 mm}
We begin with some notation in~\S\ref{S:notation} and a somewhat technical
combinatorial result, Main Lemma~\pref{ML:min_order}, which will be used in
the proof of the main theorem in the last section. We recall the definition of
polynomial laws in~\S\ref{S:pol_laws}, the algebra of divided powers
$\Gamma_A(M)$ in~\S\ref{S:divided_powers}, and the multiplicative structure of
$\Gamma^d_A(B)$ in~\S\ref{S:mul_div_powers}.
In the rest of the article we will only consider $\Gamma^d_A(B)$ for $B$
free over $A$ and as mentioned above, in this case $\Gamma^d_A(B)$ coincides
with the symmetric tensors $\TS^d_A(B)$,
see~\pref{X:Gamma-TS-as-alg}.
The reader, if inclined so, may replace $\Gamma^d_A(B)$ with $\TS^d_A(B)$,
$\gamma^k_A(x)$ with $x^{\otimes_A k}$, interpret $\times$ as the shuffle
product and forget about divided powers altogether. The important results
of~\S\S\ref{S:divided_powers}-\ref{S:mul_div_powers} used in the sequel are
Formula~\pref{F:Gamma_mult_form} describing the multiplication in
$\Gamma^d_A(B)$ and the surjective
homomorphism $\surjmap{\rho^e_d}{\Gamma^e_A(B)}{\Gamma^d_A(B)}$ for $e\geq d$
defined in paragraph~\pref{X:rho^e_d}. The homomorphism
$\rho^e_d$ allows us to lift relations in $\Gamma^d_A(B)$ to relations
in $\Gamma^e_A(B)$ which will be useful in \S\ref{S:gen_of_Gamma}. We also
use the convenient shorthand notation $\gammaid{k}$ for $\gamma^k_A(1)$ or
$1^{\otimes_A k}$.

In \S\ref{S:multi_symmetric} we establish some notation
and well known facts about the multisymmetric functions
$\Gamma^d_A(A[x_1,x_2,\dots,x_r])=\TS^d_A(A[x_1,x_2,\dots,x_r])$.
In the central section~\S\ref{S:gen_of_Gamma} a \emph{minimal} set
of generators for the ring of multisymmetric functions
$\Gamma^d_A(A[x_1,x_2,\dots,x_r])$ is found in
Theorem~\pref{T:Gamma_generators} for an arbitrary ring $A$. Several
applications of this theorem is then given in~\S\ref{S:applications}.
A sharp bound on the total degree of any generating set is given in
Corollary~\pref{C:sharp_total_degree_bound},
improving~\cite[Thm. 4.6]{fleischmann}. In Corollary~\pref{C:elementary_gen},
the cases when $\Gamma^d_A(A[x_1,x_2,\dots,x_r])$ is generated by the
elementary multisymmetric polynomials is determined as has previously been done
by Briand~\cite[Thm. 1]{briand_elem_multi_gens}.

Finally we briefly discuss the relation between the generators of the ring of
multisymmetric polynomials and the Chow scheme in Remark~\pref{R:Sym-Chow}.

\begin{acknowledgement}[\!\!\!]
I would like to thank D.~Laksov for introducing me to divided powers and
for many suggestions and comments. Thanks also to F.~Vaccarino for reading an
early manuscript and pointing out some references. Finally I thank the referee
for several valuable comments.
\end{acknowledgement}

\end{section}

\begin{section}{Multi-indices, multinomials and a combinatorial result}
\label{S:notation}

\begin{notation}
We let $\N$ be the set of non-negative integers $0,1,2,\dots$.
\end{notation}

\begin{notation}
For a multi-index $\nu\in \N^{(\I)}$ we let
$$((\nu))=\binom{|\nu|}{\nu}=
\frac{\left(\sum_\alpha \nu_{\alpha}\right)!}
{\prod_\alpha \left(\nu_{\alpha}!\right)}$$
be the multinomial coefficient of $\nu$. In particular
$((a,b))=\binom{a+b}{a}$.
\end{notation}

\begin{notation}
For $i\in \I$ we let $\basis{i}\in\N^{(\I)}$ be the multi-index such that
$$\left(\basis{i}\right)_\alpha=\begin{cases}
0 & \quad\text{if $\alpha\neq i$}\\
1 & \quad\text{if $\alpha=i$.}
\end{cases}$$
\end{notation}

\begin{definition}
For $n\in \Q$ and $p$ a prime we let $\ord_p(n)$ be the order of $n$ at
$p$, i.e. $\ord_p(n)$ is defined by
$n=\pm\prod_{p\text{ prime}} p^{\ord_{p}(n)}$.
\end{definition}

\begin{definition}\label{D:Qp}
For $n\in\N$ we let $\Qp{n}(t)=a_{s}t^s+a_{s-1}t^{s-1}+\dots+a_0\in \Z[t]$ be
the polynomial with coefficients $0\leq a_{k}\leq p-1$ such that $\Qp{n}(p)=n$,
i.e. $a_s a_{s-1}\dots a_0$ is the presentation of $n$ in base $p$. If
$\alpha\in \N^n$ then we let $\Qp{\alpha}=\sum_{k=1}^n \Qp{\alpha_k}$.
\end{definition}

\begin{definition}\label{D:comp_pols}
If $P(t),Q(t)\in\Z[t]$ are polynomials then $P>Q$ means that
$P(n)>Q(n)$ for all sufficiently large $n$.
\end{definition}

\begin{lemma}\label{L:Qp_and_multinomials}
Let $n\in \N$ and $\alpha\in \N^{(\I)}$. We have that
\begin{enumerate}
\item $\ord_p(n!)=\frac{1}{p-1}\bigl(n-\Qp{n}(1)\bigr)$.\label{EN:QP_id_2}
\item $\ord_p\;((\alpha))=
     \frac{1}{p-1}\bigl(\Qp{\alpha}(1)-\Qpp{|\alpha|}(1)\bigr)$.
\label{EN:QP_id_3}
\item $\ord_p\;((p^s\alpha))=\ord_p\;((\alpha))$.\label{EN:QP_id_4}
\end{enumerate}
\end{lemma}
\begin{proof}
\enumref{EN:QP_id_2} is easily verified,
\enumref{EN:QP_id_3} is an immediate consequence of \enumref{EN:QP_id_2} and
\enumref{EN:QP_id_4} follows from \enumref{EN:QP_id_3}.
\end{proof}

\begin{lemma}\label{L:ordp==0_and_Qp}
Let $\alpha\in\N^{(\I)}$. We have three inequalities
$$\ord_p\;((\alpha)) \geq 0 \quad\quad
\Qp{\alpha}(1) \geq\Qpp{|\alpha|}(1) \quad\quad
\Qp{|\alpha|} \geq \Qp{\alpha}$$
and equality in any of these inequalities holds if and only if the sum
$\sum_{i\in\I} \alpha_i$ can be computed in base $p$ without carrying, i.e
if and only if $\Qp{|\alpha|}=\Qp{\alpha}$.
\end{lemma}
\begin{proof}
The first inequality is obvious as $((\alpha))$ is a positive integer.
It is further easily seen that the two last inequalities hold and with
equality if and only if the sum $\sum_{i\in\I} \alpha_i$ can be computed in
base $p$ without carrying. The second inequality is a multiple of the
first by Lemma~\pref{L:Qp_and_multinomials}~\enumref{EN:QP_id_3} and thus
$\ord_p\;((\alpha))=0$ if and only if $\Qp{|\alpha|}=\Qp{\alpha}$.
\end{proof}

%%%%%%%%%%%%%%%%%%%%%%%%%%%%%%%%%%%%%%%%%%

\begin{definition}\label{D:calS}
Fix a positive integer $r$ and let $\monpos=\N^r\setminus\{0\}$. We will
identify $\beta\in\monpos$ with the monomial
$x^\beta=x_1^{\beta_1}x_2^{\beta_2}\dots x_r^{\beta_r}$. In particular,
we identify $\basis{i}\in\monpos$ with $x_i$. 
Given a multi-index
$\alpha\in\monpos$ and an integer $d<|\alpha|$ we let $\calS_{\alpha,d}$ be
the set of $\nu\in\N^{(\monpos)}$ such that there is a decomposition
$x^\alpha=\prod_{\beta\in\monpos} {\left(x^\beta\right)}^{\nu_{\beta}}$
and such that $|\nu|>d$. That is, the elements of $\calS_{\alpha,d}$ are
factorizations of $x^\alpha$ in at least $d+1$, not necessarily different,
non-trivial monomials.
\end{definition}

\begin{main_lemma}\label{ML:min_order}
Given a multi-index $\alpha\in\monpos$, an integer $d<|\alpha|$ and a
prime~$p$, we let $s=\ord_p \gcd(\alpha)$. Then there exists a
$\nu\in\calS_{\alpha,d}$ such that
$$\ord_p \frac{p^s((\nu))}{|\nu|}=0$$
if and only if $\Qp{\alpha}>\Qp{d}$.
\end{main_lemma}
\begin{proof}
\def\tnu{\widetilde{\nu}}
First note that $\ord_p\bigl(p^s((\nu))/|\nu|\bigr)$ is always non-negative.
In fact, there exists a $\beta\in\monpos$ such that $\nu_{\beta}>0$ and
$\ord_p \nu_{\beta}\leq s$, and
$((\nu))/|\nu|=((\nu-\basis{\beta}))/\nu_{\beta}$.
We will prove the lemma in several steps:

\vspace{0.3 cm}
\textbf{I)} \emph{Reduction to the case where $s=0$.}
Assume that $s>0$ and let $\nu\in\calS_{\alpha,d}$. If $p$ does not
divide~$\nu$ then define $\nu',\nu'',\tnu\in\N^{(\monpos)}$ by
$$\nu'_{\beta}=\floor{\frac{\nu_{\beta}}{p}} \quad\text{and}\quad
\tnu=p(\nu'+\nu'')
$$
where $\nu''$ is chosen such that
$$|\tnu|=p\ceil{\frac{|\nu|}{p}} \quad\text{and}\quad
\sum_{\beta\in\monpos}\tnu_{\beta}\beta=
\sum_{\beta\in\monpos}\nu_{\beta}\beta=\alpha.
$$
Then $\tnu\in\calS_{\alpha,d}$ since $|\tnu|\geq|\nu|\geq d+1$.
To see that such a $\nu''$ exists, let
$\alpha''=\alpha/p-\sum_{\beta\in\monpos}\nu'_{\beta} \beta$
and $q=\ceil{(|\nu|-p|\nu'|)/p}\leq |\alpha''|$.
Then choose $\beta_1,\beta_2,\dots,\beta_q\in\monpos$
such that $\alpha''=\beta_1+\beta_2+\dots+\beta_q$ and let
$\nu''=\sum_i \basis{\beta_i}$.
Now as $\ord_p \bigl(|\nu|!/|\nu|\bigr)=
\ord_p \bigl(|\tnu|!/|\tnu|\bigr)$ and
$\ord_p(\nu!)=\ord_p\bigl( (p\nu')!\bigr)\leq
\ord_p(\tnu!)$ we obtain that
$$\ord_p \frac{((\nu))}{|\nu|} =
\ord_p \frac{|\nu|!}{|\nu|\nu!}
\geq \ord_p \frac{|\tnu|!}{|\tnu|\tnu!}=
\ord_p \frac{((\tnu))}{|\tnu|}$$
and thus
$$\min_{\nu\in \calS_{\alpha,d}} \ord_p \frac{p^s((\nu))}{|\nu|}
=\min_{p\mu\in \calS_{\alpha,d}} \ord_p \frac{p^s((p\mu))}{|p\mu|}
=\min_{p\mu\in \calS_{\alpha,d}}
\ord_p \frac{p^{s-1}((\mu))}{|\mu|}
$$
by Lemma~\pref{L:Qp_and_multinomials} \enumref{EN:QP_id_4}. Let
$p\mu\in\calS_{\alpha,d}$. If we let $\alpha'=\alpha/p$ and
$d'=\floor{d/p}$ then $\mu\in\calS_{\alpha',d'}$ 
as $p|\mu|\geq d+1$ implies that $|\mu|\geq \ceil{(d+1)/p}=d'+1$. Thus
$$\min_{p\mu\in \calS_{\alpha,d}} \ord_p \frac{p^{s-1}((\mu))}{|\mu|}
=\min_{\mu\in \calS_{\alpha',d'}} \ord_p \frac{p^{s-1}((\mu))}{|\mu|}.
$$

Finally $\Qp{\alpha}>\Qp{d}$ if and only if $\Qp{\alpha'}>\Qp{d'}$ and
we conclude step I) by induction on $s$.

\vspace{0.3 cm}
\textbf{II)} \emph{Reduction to the case where
$\nu=\sum_{i=1}^r \delta_i \basis{x_i}+\basis{x^\gamma}$ with $\delta\in \N^r$
and $\gamma\in \monpos$.}
From step I) we can assume that $p\notdivide \alpha$ and hence
$p\notdivide \nu$. If we choose $\beta_0$ such that
$p\notdivide \nu_{\beta_0}$ then
$$\ord_p\frac{((\nu))}{|\nu|}=\ord_p\; ((\nu-\basis{x^{\beta_0}})).$$
For every $\beta\in\monpos$ choose
$i(\beta)\in\{1,2,\dots,r\}$ such that $\beta_{i(\beta)}>0$ and let
$$\nu'=\sum_{\beta\in\monpos} \nu_{\beta} \basis{x_{i(\beta)}}
-\basis{x_{i(\beta_0)}}.$$
Then $|\nu'|=|\nu|-1$ and
$\ord_p\;((\nu'))\leq \ord_p\;((\nu-\basis{x^{\beta_0}}))$. Finally
if we let $\nu''=\nu'+\basis{x^\gamma}$ with $x^\gamma\in \monpos$ such that
$\sum_{\beta\in\monpos}\nu''_{\beta}\beta=
\sum_{\beta\in\monpos}\nu_{\beta}\beta=\alpha$
then $|\nu''|=|\nu|$ and
$$\ord_p\frac{((\nu''))}{|\nu''|}=
\ord_p\frac{((\nu''-\basis{x^\gamma}))}{\nu''_\gamma}\leq
\ord_p\;((\nu'))\leq \ord_p\frac{((\nu))}{|\nu|}.$$

\vspace{0.3 cm}
Let $\calT_{\alpha,d}=\{\delta\in\N^r\;:\; \delta<\alpha,\; |\delta|\geq d\}$.

\vspace{0.3 cm}
\textbf{III)} \emph{$\ord_p\;((\nu))/|\nu|=\ord_p\;((\delta))$ for
some $\delta\in\calT_{\alpha,d}$.} From
the previous step we can assume that
$\nu=\sum_{i=1}^r \delta_i \basis{x_i}+\basis{x^\gamma}$. If
$|\gamma|\geq 2$, i.e. $x^\gamma\neq x_j$ for some $j$ then
$$\frac{((\nu))}{|\nu|}
=\frac{((\delta_1,\delta_2,\dots,\delta_r,1))}
{\delta_1+\delta_2+\dots+\delta_r+1}
=((\delta)).$$
Otherwise $x^\gamma=x_j$ and $\nu=\sum_{i=1}^r \alpha_i \basis{x_i}$.
As $p\notdivide \alpha$ by step I), there is an $k$ such that
$p\notdivide \alpha_k$ and we can write
$\nu=\sum_{i=1}^r \delta'_i \basis{x_i}+\basis{x_k}$
where $\delta'=\alpha-\basis{k}$. Then
$$\ord_p \frac{((\nu))}{|\nu|}
=\ord_p\; ((\nu-\basis{x_k}))
=\ord_p\; ((\delta')).$$

\vspace{0.3 cm}
\textbf{IV)} \emph{$\min_{\nu\in\calS_{\alpha,d}} \ord_p\;((\nu))/|\nu|=
\min_{\delta\in\calT_{\alpha,d}} \ord_p\;((\delta))$.} From step III) it
follows that
$$\min_{\nu\in\calS_{\alpha,d}} \ord_p\frac{((\nu))}{|\nu|}\geq
\min_{\delta\in\calT_{\alpha,d}} \ord_p\;((\delta)).$$
For any $\delta\in\calT_{\alpha,d}$ we let $\gamma=\alpha-\delta$ and
$\nu=\sum_i \delta_i \basis{x_i}+\basis{x^\gamma}\in\calS_{\alpha,d}$. As
$\ord_p\; ((\nu))/|\nu|=\ord_p\;((\delta))-\ord_p(\nu_{\gamma})
\leq \ord_p\;((\delta))$ this concludes step IV).

\vspace{0.3 cm}
\textbf{V)} \emph{The minimum of $\bigl\{\ord_p\;((\delta))\bigr\}_{\delta\in\calT_{\alpha,d}}$ is attained for a $\delta$ with $|\delta|=d$.}
This follows immediately from the fact that
$\ord_p\;((\delta))\geq\ord_p\;((\delta-\basis{i}))$ if $i$ is chosen such
that $\ord_p(\delta_i)$ is minimal.

\textbf{VI)} \emph{Conclusion.} Let $\delta\in\N^r$ such that
$|\delta|=d$ and $\delta<\alpha$. Lemma~\pref{L:ordp==0_and_Qp} shows that
$\ord_p\;((\delta))=0$ if and only if $\Qp{\delta}=\Qpp{|\delta|}=\Qp{d}$. It
is then easily seen that
there exists a $\delta<\alpha$ such that $\Qp{\delta}=\Qp{d}$
if and only if $\Qp{\alpha}>\Qp{d}$.
\end{proof}

\end{section}

%%%%%%%%%%%%%%%%%%%%%%%%%%%%%%%%%%%%%%%%%%

\begin{section}{Polynomial laws and symmetric tensors}\label{S:pol_laws}

We recall the definition of a polynomial
law~\cite{roby_lois_pol, roby_lois_pol_mult}.

\begin{definition}
Let $M$ and $N$ be $A$-modules. We denote by $\fF_M$ the functor
$$\map{\fF_M}{A\xAlg}{\Sets},\quad\quad A'\mapsto M\otimes_A A'.$$
A \emph{polynomial law} from $M$ to $N$ is a natural transformation
$\map{f}{\fF_M}{\fF_N}$. More concretely, a polynomial law is a \emph{map}
$\map{f_{A'}}{{M\otimes_A A'}}{{N\otimes_A A'}}$ for every
$A$-algebra $A'$ such that for any homomorphism of $A$-algebras
${\map{g}{A'}{A''}}$ the diagram
$$\xymatrix{
{M\otimes_A A'}\ar[r]^{f_{A'}}\ar[d]_{\id{M}\otimes g} & 
   {N\otimes_A A'}\ar[d]^{\id{N}\otimes g} \\
{M\otimes_A A''}\ar[r]^{f_{A''}} & {N\otimes_A A''}\ar@{}[ul]|{\circ}
}$$
commutes.
The polynomial law $f$ is
\emph{homogeneous of degree} $d$ if for any $A$-algebra $A'$, the corresponding
map $\map{f_{A'}}{M\otimes_A A'}{N\otimes_A A'}$ is such that
$f_{A'}(ax)=a^d f_{A'}(x)$ for any $a\in A'$ and $x\in M\otimes_A A'$. If $B$
and $C$ are $A$-algebras then a polynomial law from $B$ to $C$ is
\emph{multiplicative} if for any $A$-algebra $A'$, the corresponding map
$\map{f_{A'}}{B\otimes_A A'}{C\otimes_A A'}$ is such that
$f_{A'}(xy)=f_{A'}(x)f_{A'}(y)$ for any $x,y\in B\otimes_A A'$.
\end{definition}

\begin{notation}
Let $A$ be a ring and $M$ and $N$ be $A$-modules (resp. $A$-algebras).
We let $\Hompl{d}(M,N)$ (resp. $\Hommpl{d}(M,N)$) denote the polynomial laws
(resp. multiplicative polynomial laws) $M\to N$ which are homogeneous of
degree $d$.
\end{notation}

\begin{notation}
Let $A$ be a ring and $M$ an $A$-algebra. We denote the $d$\textsuperscript{th}
tensor product of $M$ over $A$ by $\T^d_A(M)$. We have an action of the
symmetric group $\SG{d}$ on $\T^d_A(M)$ permuting the factors. The invariant
ring of this action is the set of symmetric tensors and is denoted
$\TS^d_A(M)$. By $\T_A(M)$ and $\TS_A(M)$ we denote the graded $A$-modules
$\bigoplus_{d\geq 0} \T^d_A(M)$ and $\bigoplus_{d\geq 0}\TS^d_A(M)$
respectively.
\end{notation}

\begin{xpar}[Shuffle product]\label{X:shuffle_product}
When $B$ is an $A$-algebra, then $\TS^d_A(B)$ has a natural $A$-algebra
structure induced from the $A$-algebra structure of $\T^d_A(B)$. The
multiplication on $\TS^d_A(B)$ will be written as juxtaposition. For any
$A$-module $M$, we can equip $\T_A(M)$ and $\TS_A(M)$ with $A$-algebra
structures compatible with the grading. The multiplication on $\T_A(M)$ is the
ordinary tensor product and the multiplication on $\TS_A(M)$ is called the
\emph{shuffle product} and is denoted by $\times$. If $x\in\TS^d_A(M)$ and
$y\in\TS^e_A(M)$ then
$$x\times y=\sum_{\sigma\in\SG{d,e}} \sigma\left(x\otimes_A y\right)$$
where $\SG{d,e}$ is the subset of $\SG{d+e}$ such that
$\sigma(1)<\sigma(2)<\dots<\sigma(d)$ and
$\sigma(d+1)<\sigma(d+2)<\dots\sigma(d+e)$.
\end{xpar}
\end{section}

\vspace{-5 mm} % AIF
\begin{section}{Divided powers}\label{S:divided_powers}
All of the material in this section can be found in~\cite{roby_lois_pol}
and a good exposition is~\cite[\S2]{ferrand_norme}.

\begin{xpar}
Let $A$ be a ring and $M$ an $A$-module. Then there exists a graded
$A$-algebra, the algebra of divided powers, denoted
$\Gamma_A(M)=\bigoplus_{d\geq 0}\Gamma^d_A(M)$ equipped with maps
$\map{\gamma^d}{M}{\Gamma^d_A(M)}$ such that, denoting the multiplication with
$\times$ as in~\cite{ferrand_norme}, we have that for every $x,y\in M$,
$a\in A$ and $d,e\in\N$
\begin{align}
\Gamma^0_A(M) & = A,\quad\text{and}\quad \gamma^0(x)=1 \label{gamma0}\\
\Gamma^1_A(M) & = M,\quad\text{and}\quad \gamma^1(x)=x \label{gamma1}\\
\gamma^d(ax) & = a^d\gamma^d(x) \\
\gamma^d(x+y) & = \textstyle\sum_{d_1+d_2=d} \gamma^{d_1}(x)\times \gamma^{d_2}(y) \\
\gamma^d(x)\times\gamma^e(x) & = ((d,e))\gamma^{d+e}(x).
\end{align}
Using~\eqref{gamma0} and~\eqref{gamma1} we will identify $A$ with
$\Gamma^0_A(M)$ and $M$ with $\Gamma^1_A(M)$.
If $(x_\alpha)_{\alpha\in\I}$ is a family of elements of $M$ and
$\nu\in\N^{(\I)}$ then we let
$$\gamma^\nu (x) = \bigtimes_{\alpha\in \I} \gamma^{\nu_{\alpha}} (x_\alpha)$$
which is an element of $\Gamma^d_A(M)$ with
$d=|\nu|=\sum_{\alpha\in\I} \nu_{\alpha}$.
\end{xpar}

\begin{xpar}[Functoriality]
$\Gamma_A(\cdot)$ is a covariant functor from the category of $A$-modules to
the category of graded $A$-algebras~\cite[Ch. III \S4, p. 251]{roby_lois_pol}.
\end{xpar}

\begin{xpar}[Base change]
For any $A$-algebra $A'$ there is a natural isomorphism
$\Gamma_{A'}(M\otimes_A A')\to \Gamma_A(M)\otimes_A A'$ taking
${\gamma^d(x\otimes_A 1)}$ to ${\gamma^d(x)\otimes_A 1}$~%
\cite[Thm. III.3, p. 262]{roby_lois_pol}.
This shows that $\gamma^d$ is a homogeneous polynomial law of degree $d$.
\end{xpar}

\begin{xpar}[Universal property]\label{X:univ_prop_of_Gamma}
The map $\Hom_A\bigl(\Gamma^d_A(M),N\bigr)\to\Hompl{d}(M,N)$ given by
$f\mapsto f\circ\gamma^d$ is an isomorphism~%
\cite[Thm. IV.1, p. 266]{roby_lois_pol}.
\end{xpar}

\begin{xpar}[Basis]\label{X:Gamma_basis}
If $(x_\alpha)_{\alpha\in\I}$ is a family of elements of $M$ which generates
$M$, then the family
$\bigl(\gamma^\nu(x)\bigr)_{\nu\in\N^{(\I)}}$ generates $\Gamma_A(M)$. If
$(x_\alpha)_{\alpha\in\I}$ is a basis of $M$
then $\bigl(\gamma^\nu(x)\bigr)_{\nu\in\N^{(\I)}}$ is a basis of
$\Gamma_A(M)$~\cite[Thm. IV.2, p. 272]{roby_lois_pol}.
\end{xpar}

\begin{xpar}[Exactness]\label{X:Gamma_exactness}
The functor $\Gamma_A(\cdot)$ is a left
adjoint~\cite[Thm. III.1, p. 257]{roby_lois_pol} and thus commutes with any
(small) direct limit. It is thus
\emph{right exact}~\cite[Def. 2.4.1]{sga4_grothendieck_prefaisceaux} but note
that $\Gamma_A(\cdot)$ is a functor from the category of $A$-modules to the
category of graded $A$-algebras and that the latter category is not abelian.
By~\cite[Rem. 2.4.2]{sga4_grothendieck_prefaisceaux} a functor is right
exact if and only if it takes the initial object onto the initial object and
commutes with finite coproducts and coequalizers. Thus
$\Gamma_A(0)=A$ and given an exact diagram of $A$-modules
$$\xymatrix{
{M'} \ar@<.5ex>[r]^f \ar@<-.5ex>[r]_g & {M} \ar[r]^h & {M''}
}$$
the diagram
$$\xymatrix{
{\Gamma_A(M')}\ar@<.5ex>[r]^{\Gamma f} \ar@<-.5ex>[r]_{\Gamma g} &
{\Gamma_A(M)}\ar[r]^{\Gamma h} & {\Gamma_A(M'')}
}$$
is an exact sequence of $A$-algebras and
$$\Gamma_A(M\oplus M')=\Gamma_A(M)\otimes_A \Gamma_A(M').$$
The latter identification can be made
explicit~\cite[Thm. III.4, p. 262]{roby_lois_pol} as
\begin{align*}
\Gamma^d_A(M\oplus M') &= \bigoplus_{a+b=d}
     \left(\Gamma^a_A(M)\otimes_A \Gamma^b_A(M')\right) \\
\gamma^d(x+y) &= \sum_{a+b=d}\gamma^a(x)\otimes\gamma^b(y).
\end{align*}
This makes $\Gamma_A(M\oplus M')=
\bigoplus_{a,b\geq 0}
\Gamma^{a,b}_A(M\oplus M')$ into a bigraded algebra where
$\Gamma^{a,b}_A(M\oplus M')=\Gamma^a_A(M)\otimes_A\Gamma^b_A(M')$.
\end{xpar}

\begin{xpar}[Exactness of $\Gamma^d_A(\cdot)$]
If $M\surj N$ is a surjection then it is easily seen
from the explicit generators of $\Gamma^d(N)$ in~\pref{X:Gamma_basis} that
$\Gamma^d_A(M)\surj\Gamma^d_A(N)$ is surjective.
This does, however, not imply that $\Gamma^d_A(\cdot)$ is
right exact. In fact, in general it is not since we have that
$\Gamma^d_A(M\oplus M')\neq \Gamma^d_A(M)\oplus\Gamma^d_A(M')$.
\end{xpar}

\begin{xpar}[Filtered direct limits]\label{X:Gamma^d_filt_dir_lims}
The functor $\Gamma^d_A(\cdot)$ commutes with \emph{filtered} direct limits.
In fact, if $(M_\alpha)$ is a directed filtered system of $A$-modules then
\begin{align*}
\bigoplus_{d\geq 0} \Gamma^d_A
  (\varinjlim_{\!\!\scriptscriptstyle A\xMod\!\!} M_\alpha)
&= \varinjlim_{\scriptscriptstyle A\xAlg}
  \bigoplus_{d\geq 0}\Gamma^d_A(M_\alpha) = \\
&= \varinjlim_{\scriptscriptstyle A\xMod}
  \bigoplus_{d\geq 0}\Gamma^d_A(M_\alpha) =
\bigoplus_{d\geq 0}\varinjlim_{\scriptscriptstyle A\xMod}\Gamma^d_A(M_\alpha).
\end{align*}
The first equality follows from the exactness of $\Gamma$ described in
paragraph~\pref{X:Gamma_exactness} and the second
from the fact that a filtered direct limit in the category of $A$-algebras
coincides with the corresponding filtered direct limit in the category of
$A$-modules~\cite[Cor. 2.9]{sga4_grothendieck_prefaisceaux}.
\end{xpar}

\begin{xpar}[$\Gamma$ and $\TS$]\label{X:Gamma-TS-as-mod}
The homogeneous polynomial law $M\to\TS^d_A(M)$ of degree $d$ given by
$x\mapsto x^{\otimes_A d} = x\otimes_A \dots\otimes_A x$ corresponds by the
universal property~\pref{X:univ_prop_of_Gamma} to an $A$-module homomorphism
$\Gamma^d_A(M)\to \TS^d_A(M)$. This extends to an $A$-algebra homomorphism
$\Gamma_A(M)\to \TS_A(M)$, where the multiplication in $\TS_A(M)$ is the
shuffle product defined in paragraph~\pref{X:shuffle_product},
cf.~\cite[Prop. III.1, p.254]{roby_lois_pol}.

When $M$ is a free $A$-module the homomorphisms $\Gamma^d_A(M)\to \TS^d_A(M)$
and $\Gamma_A(M)\to \TS_A(M)$ are isomorphisms of $A$-modules
respectively $A$-algebras~\cite[Prop. IV.5, p. 272]{roby_lois_pol}.
More generally, these homomorphisms are isomorphisms when $M$ is a \emph{flat}
$A$-module, see~\cite[5.5.2.5, p. 123]{sga4_deligne_coh_supp_prop}, or when $A$
is a $\Q$-algebra, see~\cite[Ch. III, Cor., p. 256]{roby_lois_pol}. This is
also discussed in~\cite{rydh_gammasymchow_inprep}.
\end{xpar}

\begin{xpar}\label{X:splitting_hom_of_gamma_d_1+d_2}
Let $d_1,d_2\in \N$. There is a canonical homomorphism
$$\map{\rho^{d_1+d_2}_{d_1,d_2}}{\Gamma^{d_1+d_2}_A(M)}
{\Gamma^{d_1}_A(M)\otimes_A\Gamma^{d_2}_A(M)}$$
given by the homogeneous polynomial law
${x\mapsto\gamma^{d_1}(x)\otimes\gamma^{d_2}(x)}$
of degree $d_1+d_2$ and the universal property~\pref{X:univ_prop_of_Gamma}.
In particular
$$\rho^{d_1+d_2}_{d_1,d_2}\bigl(\gamma^{\nu}(x)\bigr)=
\sum_{\substack{\nu^{(1)}+\nu^{(2)}=\nu\\\left|\nu^{(i)}\right|=d_i}}
\gamma^{\nu^{(1)}}(x)\otimes \gamma^{\nu^{(2)}}(x).$$
\end{xpar}
\end{section}

\begin{section}{Multiplicative structure of $\Gamma^d_A(B)$}\label{S:mul_div_powers}
When $B$ is a not necessarily commutative and unitary $A$-algebra then the
multiplication of $B$ induces a multiplication on $\Gamma^d_A(B)$ which we will
denote by juxtaposition~\cite[(II)]{roby_lois_pol_mult}. In particular
$\gamma^d(x)\gamma^d(y)=\gamma^d(xy)$ and this makes
$\gamma^d$ into a multiplicative polynomial law homogeneous of degree~$d$.
Unless otherwise stated we will assume that $B$ is a (commutative and unitary)
$A$-algebra. The unit in $\Gamma^d_A(B)$ is $\gamma^d(1)$ and will be denoted
by $\gammaid{d}$ in the sequel.

\begin{xpar}[Universal property]
Let $B$ and $C$ be $A$-algebras. Then the map
$\Hom_{A\xAlg}\bigl(\Gamma^d_A(B),C\bigr)\to\Hommpl{d}(B,C)$ given by
$f\to f\circ\gamma^d$ is an isomorphism~\cite[Thm.]{roby_lois_pol_mult}. Also
see~\cite[Prop. 2.5.1]{ferrand_norme}.
\end{xpar}

\begin{xpar}[$\Gamma$ and $\TS$]\label{X:Gamma-TS-as-alg}
The homogeneous polynomial law $B\to\TS^d_A(B)$ of degree $d$ given by
$x\mapsto x^{\otimes_A d} = x\otimes_A \dots\otimes_A x$ is multiplicative.
The homomorphism $\map{\varphi}{\Gamma^d_A(B)}{\TS^d_A(B)}$
in paragraph~\pref{X:Gamma-TS-as-mod} is thus an $A$-algebra homomorphism. It
is an isomorphism when $B$ is a free $A$-module, or more generally when $B$ is
flat over $A$, since it is then an isomorphism of $A$-modules
by~\pref{X:Gamma-TS-as-mod}.

This $A$-algebra homomorphism is studied in~\cite{rydh_gammasymchow_inprep} and
need neither be injective nor surjective.
In particular it is shown that if $x\in \ker \varphi$ then $x^{d!}=0$ and if
$y\in \TS^d_A(B)$ then $y^{d!}\in \image \varphi$. Further, the corresponding
morphism of schemes
$$\Sym^d_{\Spec(A)}\bigl(\Spec(B)\bigr):=
\Spec\bigl(\TS^d_A(B)\bigr)\to\Spec\bigl(\Gamma^d_A(B)\bigr)$$
is a universal homeomorphism. An example due to
Lundkvist~\cite{lundkvist_counter-ex} shows that
$\bigl(\Gamma^d_A(B)\bigr)_{\red}\to \bigl(\TS^d_A(B)\bigr)_{\red}$ is not
always an isomorphism. The induced morphism between
\emph{seminormalizations} $\bigl(\Gamma^d_A(B)\bigr)_{\mathrm{sn}}\to
\bigl(\TS^d_A(B)\bigr)_{\mathrm{sn}}$, however, is an isomorphism as shown
in~\cite{rydh_gammasymchow_inprep}.
\end{xpar}

\begin{formula}[Multiplication formula {\cite[Eq. (3)]{roby_lois_pol_mult}}]
\label{F:Gamma_mult_form}
Let $(x_\alpha)_{\alpha\in\I}$ be a family of elements in $B$ and let
$\mu,\nu\in \N^{(\I)}$ with $d=|\mu|=|\nu|$. Then we have the following
identity in $\Gamma^d_A(B)$
$$\gamma^\mu(x)\gamma^\nu(x)=
\sum_{\xi\in N_{\mu,\nu}} \gamma^\xi(x_{(1)}x_{(2)})=
\sum_{\xi\in N_{\mu,\nu}} \bigtimes_{(\alpha,\beta)\in\I\times\I}
       \gamma^{\xi_{\alpha\beta}}(x_\alpha x_\beta)
$$
where $N_{\mu,\nu}$ is the set of multi-indices $\xi\in \N^{(\I\times\I)}$
such that $\sum_{\beta\in\I} \xi_{\alpha\beta}=\mu_{\alpha}$ for every
$\alpha\in \I$ and $\sum_{\alpha\in\I} \xi_{\alpha\beta}=\nu_{\beta}$ for
every $\beta\in\I$.
\end{formula}

\begin{xpar}\label{X:mult_splitting_hom_of_gamma_d_1+d_2}
The homomorphism $\rho^{d_1+d_2}_{d_1,d_2}$
of~\pref{X:splitting_hom_of_gamma_d_1+d_2} was given by the homogeneous
polynomial law $B\to \Gamma^{d_1}_A(B)\otimes_A\Gamma^{d_2}_A(B)$ defined by
${x\mapsto\gamma^{d_1}(x)\otimes\gamma^{d_2}(x)}$. As this is a
multiplicative law we get a homomorphism of $A$-algebras
$$\map{\rho^{d_1+d_2}_{d_1,d_2}}{\Gamma^{d_1+d_2}_A(B)}
{\Gamma^{d_1}_A(B)\otimes_A\Gamma^{d_2}_A(B)}.$$
Further, given an $A$-algebra retraction $\map{\epsilon}{B}{A}$ of the
structure homomorphism $A\to B$ we get a homomorphism
$$\Gamma^{d_1+d_2}_A(B)\to
\Gamma^{d_1}_A(B)\otimes_A\Gamma^{d_2}_A(B)\to
\Gamma^{d_1}_A(B)\otimes_A\Gamma^{d_2}_A(A)\iso \Gamma^{d_1}_A(B)$$
which we will denote $\rho^{d_1+d_2}_{d_1,\epsilon}$.
\end{xpar}

\begin{xpar}\label{X:rho^e_d}
If $B=\bigoplus_{k\geq 0} B_k$ is a graded $A$-algebra with $B_0=A$ then we
have a canonical augmentation $B\to B_0=A$. The corresponding homomorphism
$\Gamma^{d_1+d_2}_A(B)\to \Gamma^{d_1}_A(B)$ given
by~\pref{X:mult_splitting_hom_of_gamma_d_1+d_2} will be denoted
$\rho^{d_1+d_2}_{d_1}$. If $\left(x_{\alpha}\right)_{\alpha\in \I}$ is a
family of generators of $B_+=\bigoplus_{k\geq 1} B_k$ as an $A$-module, then
by~\pref{X:Gamma_basis} the family $\bigl(\gamma^\nu(x)\times
\gammaid{d-|\nu|}\bigr)_{\nu\in\N^{(\I)}, |\nu|\leq d}$ generates
$\Gamma^d_A(B)$. If $e\geq d$
$$\rho^e_d\bigl(\gamma^{\nu}(x)\times\gammaid{e-|\nu|}\bigr)=
\begin{cases}
\gamma^{\nu}(x)\times\gammaid{d-|\nu|} & \text{if $|\nu|\leq d$}\\
0 & \text{if $|\nu|> d$}.
\end{cases}$$
Thus $\rho^e_d$ is surjective.
\end{xpar}

\begin{remark}[Geometrical interpretation of $\rho$]
If $A=k$ is an algebraically closed field, then the $k$-points of
$\Spec\bigl(\Gamma^d_k(B)\bigr)=\Sym^d_k\bigl(\Spec(B)\bigr)$ correspond to the
zero-cycles of degree $d$ on $\Spec(B)$. Similarly, for any reduced $A'$ it is
possible to describe the $A'$-points of $\Spec\bigl(\Gamma^d_A(B)\bigr)$ as
families of zero-cycles of degree $d$ on $\Spec(B)$ parameterized
by $\Spec(A')$~\cite{rydh_representability_of_gamma_inprep}. The homomorphism
$\rho^{d_1+d_2}_{d_1,d_2}$ defined
in~\pref{X:mult_splitting_hom_of_gamma_d_1+d_2} corresponds to a morphism of
schemes
$$\Spec\bigl(\Gamma^{d_1}_A(B)\bigr)\times_A\Spec\bigl(\Gamma^{d_2}_A(B)\bigr)
\rightarrow \Spec\bigl(\Gamma^{d_1+d_2}_A(B)\bigr)$$
describable as the addition of two families of cycles of degrees
$d_1$ and $d_2$ respectively.

A retraction $\map{\epsilon}{\Gamma^1_A(B)=B}{A}$ gives a family of zero-cycles
of degree~$1$. The homomorphism $\rho^e_{d,\epsilon}$ corresponds to the
addition of a family of cycles of degree $d$ with $e-d$ times the family of
cycles corresponding to $\epsilon$. When $B=A[x_1,x_2,\dots,x_n]$ with its
natural grading, then the canonical retraction on the zeroth graded part
corresponds to the constant family of zero-cycles which is the origin in every
fiber. The homomorphism $\rho^e_d$ is the addition of a family of cycles
of degree $d$ with $e-d$ times this constant family.
\end{remark}

\end{section}

\begin{section}{Elementary symmetric polynomials and power sums}
\label{S:multi_symmetric}

\begin{xpar}
If $B=A[x]$ is the polynomial ring in one variable, then
$\Gamma^d_A(B)\iso\TS^d_A(B)$ is isomorphic to the ring of symmetric
polynomials in $d$ variables $A[x_1,x_2,\dots,x_d]^{\SG{d}}$. It is
well known that the
ring of symmetric polynomials in $d$ variables is the polynomial ring over
$A$ in the elementary symmetric functions $e_1,e_2,\dots,e_d$. The elementary
symmetric functions correspond to the elements
$e_k=\gamma^k(x)\times \gammaid{d-k}$ in $\Gamma^d_A(B)$. We also let
$e_0=1=\gammaid{d}$ and $e_k=0$ for all $k>d$.
\end{xpar}

\begin{xpar}
We have another distinguished set of symmetric polynomials in $d$ variables,
the power sums. Let $p_k=x_1^k+x_2^k+\dots+x_d^k$ for $k\in \N$. This
corresponds to the element $p_k=x^k\times \gammaid{d-1}$. Note that
$$p_0=\gammaid{1}\times \gammaid{d-1}=((1,d-1))\gammaid{d}=d.$$
Expressed in $\Gamma^d_A(B)$, the Newton identities relating $e_k$ and $p_k$
become
\begin{equation}\label{E:Newton_identities}
\sum_{\substack{a+b=k\\a>0, b\geq 0}}(-1)^b p_a e_b + (-1)^k k e_k = 0,
\quad k=1,2,\dots .
\end{equation}
By induction these give the Waring formula, expressing $p_k$ as a polynomial
in $e_1,e_2,\dots,e_k$. Conversely, if $k!$ is invertible we obtain a formula
expressing $e_k$ in $p_1,p_2,\dots,p_k$. Thus if $A$ is purely of
characteristic $0$, i.e. a $\Q$-algebra, then
$\Gamma^d_A(B)=A[e_1,e_2,\dots,e_d]=A[p_1,p_2,\dots,p_d]$.
\end{xpar}

\begin{xpar}
Fix an integer $r\geq 1$ once and for all and let $B=A[x_1,x_2,\dots,x_r]$ be
the polynomial ring in $r$ variables. The rest of the paper will be devoted
to the study of the ring of multisymmetric polynomials
$\Gamma^d_A(B)=\TS^d_A(B)$.
\end{xpar}

\begin{xpar}
The analogues of the elementary
symmetric polynomials are the elementary multisymmetric polynomials $e_\alpha$
given by
$$e_\alpha=\gamma^\alpha(x)\times\gammaid{d-|\alpha|}=
\left(\bigtimes_{i=1}^r \gamma^{\alpha_i}(x_i)\right) \times
\gammaid{d-|\alpha|}$$
for $\alpha\in\N^r$ such that $|\alpha|\leq d$ and $e_\alpha=0$ otherwise. The
elementary multisymmetric polynomials which only depend on one set of
variables, i.e.
such that $\alpha_i=k$ for some $i$ and zero otherwise with $1\leq k\leq d$,
are called \emph{primitive}.
\end{xpar}

\begin{xpar}
Similarly as in the $1$-dimensional case, we define the multisymmetric power
sum $p_\alpha$ as
$$p_\alpha=x^\alpha\times\gammaid{d-1}=
\left(\prod_{i=1}^r x_i^{\alpha_i}\right)\times\gammaid{d-1}$$
with $\alpha\in\N^r$. As before, the $p_\alpha$'s with $\alpha_i=k$ for
some $i$ and zero otherwise with $1\leq k\leq d$, are called \emph{primitive}.
\end{xpar}

\begin{xpar}
When $r>1$ it is no longer true that $\Gamma^d_A(B)$ is a polynomial ring. It
is however easily seen that $\Gamma^d_A(B)=\TS^d_A(B)$ has relative
dimension~$rd$ over $A$, i.e. $\Sym^d_{\Spec(A)}\bigl(\Spec(B)\bigr):=
\Spec\bigl(\TS^d_A(B)\bigr)$ is equidimensional of relative dimension~$rd$
over $\Spec(A)$, cf.~\cite[Def. 13.3.2, Err. 35]{egaIV}. In fact
$\TS^d_A(B)\inj \T^d_A(B)$ is finite and the latter ring has relative
dimension~$rd$. A transcendence basis for $\Gamma^d_A(B)$ over $A$ is given
either by the primitive elementary multisymmetric functions or the primitive
multisymmetric power sums.

It is well known and often (falsely) attributed to Weyl~\cite{weyl_invariants}
that when $A$ is purely of characteristic $0$ then $\Gamma^d_A(B)$ is generated
by the $p_\alpha$'s (or the $e_\alpha$'s) with ${|\alpha|\leq d}$. This result
will also follow from Theorem~\pref{T:Gamma_generators} which gives generators
for $\Gamma^d_A(B)$ for arbitrary $A$,
cf.~Corollary~\pref{C:generators_in_char_0}. For a brief outline of the
classical proofs, see paragraph~\pref{X:gen_TS_in_char_0}.
\end{xpar}

The Newton identities have a generalization to the multisymmetric case which
has long been known, cf.~\cite[\S 4]{junker_93}. Recall that $e_\alpha=0$ if
$|\alpha|>d$.

\begin{proposition}[Multisymmetric Newton identities]\label{P:multi_symm_newton_ids}
Let $\delta\in\N^r$ then
$$\sum_{\substack{\alpha+\beta=\delta\\
                  \alpha\in\N^r\setminus 0,\; \beta\in\N^r}}
(-1)^{|\beta|} ((\alpha))
p_\alpha e_\beta + (-1)^{|\delta|} |\delta| e_\delta=0.$$
\end{proposition}
\begin{proof}
The multisymmetric identities easily follow from the usual Newton identities
by \emph{polarization}: Let $A'=A[u_1,u_2,\dots,u_r]$.
In $\Gamma^d_{A'}(A'[t])$ we have the usual Newton identities. Using the
homomorphism $A'[t]\to A'[x_1,x_2,\dots,x_r]$ defined by
$t\mapsto u_1x_1+u_2x_2+\dots+u_r x_r$ we obtain identities in
$\Gamma^d_{A'}(A'[x_1,x_2,\dots,x_r])$. Equating the coefficients of
$u^\delta$ will then give the requested identity. For details
see~\cite[Ch. II \S3]{weyl_invariants}.
% or~\cite[Prop. 1.5]{briand_thesis}.
\end{proof}

\begin{corollary}\label{C:elem_equiv_power_sums_low_deg}
If $d!$ is invertible in $A$ then the two subrings of $\Gamma^d_A(B)$,
generated by $(p_\alpha)_{|\alpha|\leq d}$ and $(e_\alpha)_{|\alpha|\leq d}$
respectively, coincide.
\end{corollary}
\begin{proof}
Repeated use of Proposition~\pref{P:multi_symm_newton_ids}
with $|\delta|\leq d$ allows us to express $p_\beta$ with $|\beta|\leq d$ as a
polynomial in $(e_\alpha)_{\alpha\leq \beta}$ and $e_\beta$ as a polynomial
in $(p_\alpha)_{\alpha\leq \beta}$. In fact, all coefficients of the involved
identities are $\leq |\delta|!\leq d!$ and hence invertible.
\end{proof}

\begin{definition}[Basis]
Let $\mon$ be the monomials in $B$ and $\monpos=\mon\setminus\{1\}$. 
For $\nu\in \N^{(\mon)}\simeq \N^{(\N^r)}$ we let
$$\msb{\nu}=\gamma^\nu(x)=\bigtimes_{\alpha\in \N^r}
\gamma^{\nu_\alpha}(x^\alpha)\in \Gamma_A(B)$$
By paragraph~\pref{X:Gamma_basis} these elements form a basis of $\Gamma_A(B)$.
A basis for $\Gamma^d_A(B)$ is then
$\bigl(\msb{\nu}\bigr)_{\nu\in\N^{(\mon)},|\nu|=d}$ or written differently
$\bigl(\msb{\nu}\times\gammaid{d-|\nu|}\bigr)_
{\nu\in\N^{(\monpos)},|\nu|\leq d}$.
\end{definition}

\begin{definition}[Multidegree]
Let the \emph{multidegree}
$\map{\mdeg}{\mon}{\N^r}$ be defined by $\mdeg(x^\alpha)=\alpha$. We
have a $\N^r$-grading on $B$ given by
$$B=\bigoplus_{\alpha\in\N^r} B_\alpha
=\bigoplus_{\alpha\in\N^r} Ax^\alpha.$$
This grading on $\Gamma^1_A(B)=B$ induces in a natural way a $\N^r$-grading on
the $A$-algebra $\Gamma_A(B)$ such that
$$\mdeg\bigl(\msb{\nu}\bigr)=\sum_{\alpha\in \N^r} \nu_{\alpha} \alpha.$$
We let $\Gamma^d_A(B)_\alpha$ be the $A$-module generated by the basis elements
$\msb{\nu}$, $|\nu|=d$ of multidegree $\alpha$. This makes
$\Gamma^d_A(B)=\bigoplus_{\alpha\in\N^r} \Gamma^d_A(B)_\alpha$ into a
$\N^r$-graded $A$-algebra as is easily seen from
Formula~\pref{F:Gamma_mult_form}.

The \emph{total degree} is the sum of the degree in every variable, e.g. the
total degree of $\msb{\nu}$ is $\sum_{\alpha\in\N^r} \nu_\alpha |\alpha|$.
\end{definition}

\begin{remark}
The multisymmetric polynomials $e_\alpha$ and $p_\alpha$ both have
multidegree $\alpha$.
\end{remark}

\begin{definition}
We let $\Gamma^d_A(B)_\ltdeg{\alpha}=
A\bigl[\Gamma^d_A(B)_{\beta,\beta<\alpha}\bigr]$ be the sub\-algebra of
$\Gamma^d_A(B)$ generated by elements of multidegree strictly less
than~$\alpha$.
\end{definition}

\begin{definition}\label{D:length}
The \emph{length} of an element $f\in\Gamma^d_A(B)$ is the smallest integer
$\ell$ such that $f=g\times\gammaid{d-\ell}$ for some $g\in\Gamma^\ell_A(B)$.
\end{definition}

\begin{definition}
Let $\primes(A)$ be the set of primes $p\in\N$ such that $p\cdot 1_A\in A$
is not invertible.
\end{definition}

\begin{remark}
If $A$ is purely of characteristic $0$, i.e. a $\Q$-algebra, then
$\primes(A)=\emptyset$. If $A$ is a local ring with residue field of
characteristic $p>0$ or $A$ is an algebra over a field of characteristic $p>0$
then $\primes(A)=\{p\}$. If $A$ is a $\Z_{(p)}$-algebra then
$\primes(A)\subseteq \{p\}$.
\end{remark}

\end{section}

%%%%%%%%%%%%%%%%%%%%%%%%%%%%%%%%%%%%%%%%%%%%%%%%%%%%%%%%%%%%%%%%%%%%%%%%%%%%

\begin{section}{Generators for the ring of multisymmetric polynomials}
\label{S:gen_of_Gamma}

As before we let $A$ be any ring, $r\geq 1$ a fix integer and
$B=A[x_1,\dots,x_r]$. In this section we will prove the main theorem of this
paper~\pref{T:Gamma_generators} in which a minimal generating set of
$\Gamma^d_A(B)\iso\TS^d_A(B)$ as an $A$-algebra is given for any ring $A$.

\begin{xpar}[Classical proof in characteristic zero]\label{X:gen_TS_in_char_0}
In characteristic zero, it can
be proved~\cite{junker_93,noether_endlichkeitssatz,weyl_invariants,neeman}
that the elementary multisymmetric functions $(e_\alpha)_{|\alpha|\leq d}$
generate the $A$-algebra $\Gamma^d_A(B)$ as follows:
\begin{enumerate}
\item[1)] Any multisymmetric function is a
polynomial in the multisymmetric power sums $(p_\alpha)_{\alpha\in \N^r}$,
see~\cite[pp. 15--16]{schlafli}, \cite[\S 5]{junker_93} or
\cite[Lemma 1]{neeman}. As any element of length $1$ is a sum of multisymmetric
power sums, this can easily be proved using induction on the length.
\item[2)] The $p_\alpha$:s can be expressed in the elementary multisymmetric
functions $(e_\alpha)_{|\alpha|\leq d}$ and vice versa,
see~\cite[\S 4]{junker_93}, \cite[Lemma 2--3]{neeman} or
Proposition~\pref{P:multi_symm_newton_ids}.
\end{enumerate}
When $r=1$ step 2 is given by the classical Newton
identities~\eqref{E:Newton_identities}.
\end{xpar}

\begin{xpar}[Proof in arbitrary characteristic]
The proof will roughly follow the same line as in characteristic zero
but a much more careful treatment is required in arbitrary characteristic.

Let $g_{k,\alpha}=\gamma^k(x^\alpha)\times\gammaid{d-k}$ where $x^\alpha\in B$
is a monomial and $1\leq k\leq d$ is an integer. 
We will first show in Proposition~\pref{P:Newton_generators} that
$\Gamma^d_A(B)$ is generated as an $A$-algebra by
$(g_{k,\alpha})_{(k,\alpha)\in\calC_0}$ where
$$\calC_0=\{(k,\alpha)\;:\; 1\leq k\leq d,\; \alpha\in\N^r\setminus 0 \}.$$
Using an analogue, Proposition~\pref{P:generalized_classical_NR}, of the Newton
identities, we will then show in Corollary~\pref{C:gen_of_Gamma_in_one_degree}
that if $\gcd(\alpha)$ is invertible in $A$ then $g_{k,\alpha}$ together
with elements of strictly smaller multidegree generate every element of
multidegree $k\alpha$. We can therefore choose a subset $\calC$ of $\calC_0$
such that every multidegree $k\alpha$ occurs once in $\calC$ and
$(g_{k,\alpha})_{(k,\alpha)\in\calC}$ generates $\Gamma^d_A(B)$:
\begin{enumerate}
\item
If every prime is invertible in $A$, we can let $\calC$ be the
subset of $\calC_0$ with $k=1$. This then gives the same generating set as
obtained in step 1) of paragraph~\pref{X:gen_TS_in_char_0}.

\item
If $k$ is a field of characteristic $p>0$ and $A$ is a $k$-algebra or a local
ring with residue field $k$,
we can choose $\calC$ as the pairs $(k,\alpha)$ with $k=p^s$, $s\in\N$ and
$p\notdivide \alpha$.

\item
For general $A$ we can choose $\calC$ as the subset of $\calC_0$ such that
$\gcd(\alpha)=1$.
\end{enumerate}

The difficult part is then to get a characterization of the multidegrees
for which the corresponding generators are generated by elements of smaller
multidegree. The main ingredient is a careful study,
Proposition~\pref{P:identities_between_elem_and_Newton}, of the possible
relations between the elementary multisymmetric functions and multiples of the
generators of Corollary~\pref{C:Newton_generators2}. This is the analogue of
the second step of~\pref{X:gen_TS_in_char_0}.
Using this ingredient and Main Lemma~\pref{ML:min_order},
an explicit minimal generator set is obtained in
Theorem~\pref{T:Gamma_generators}.
\end{xpar}

We begin with the following proposition which appears
as~\cite[Cor. 4.5 b)]{ziplies_div_pow_gens_and_trace_ids} and
\cite[Cor. 2.3]{vaccarino_gen_TS}.

\begin{proposition}\label{P:Newton_generators}
$\Gamma^d_A(B)$ is generated as an $A$-algebra by elements of the form
$\gamma^k(x^\alpha)\times \gammaid{d-k}$
where $x^\alpha$ is a monomial in $B$ and $1\leq k\leq d$.
\end{proposition}
\begin{proof}
We will use induction on the length, see Definition~\pref{D:length}, and prove
that every element of the basis
$\bigl(\msb{\nu}\times\gammaid{d-|\nu|}\bigr)_{\nu\in\N^{(\monpos)},
|\nu|\leq d}$ can be written as a sum of products of elements
$\gamma^{k}(x^\alpha)\times \gammaid{d-k}$. The length of an element
$u=\msb{\nu}\times\gammaid{d-|\nu|}$ is $\ell=|\nu|$.
If $u$ is not in the collection of the proposition then
$u=\msb{\nu_1}\times \msb{\nu_2} \times \gammaid{d-\ell}$ for some
non-zero $\nu_1,\nu_2\in\N^{(\monpos)}$ such that $\nu=\nu_1+\nu_2$.
Using Formula~\pref{F:Gamma_mult_form} we can then write
\begin{align*}
u &=\msb{\nu_1}\times\msb{\nu_2}\times
\gammaid{d-\left|\nu_1\right|-\left|\nu_2\right|}\\
&= \bigl(\msb{\nu_1}\times \gammaid{d-\left|\nu_1\right|}\bigr)
\bigl(\msb{\nu_2}\times \gammaid{d-\left|\nu_2\right|}\bigr)
-\sum_{\substack{\mu\in\N^{(\monpos)}\\|\mu|<\ell}}c_\mu
\msb{\mu}\times\gammaid{d-|\mu|}
\end{align*}
for some $c_\mu\in\N$. As this is a sum of products of terms of length
$<\ell$ we can conclude by induction.
\end{proof}

\begin{xpar}
The classical Newton identities~\eqref{E:Newton_identities} show that for
$x^\alpha\in B$ and $m\leq d$
$$\gamma^1(x^{m\alpha})\times\gammaid{d-1}
+(-1)^{m} m\gamma^m(x^\alpha)\times\gammaid{d-m} \in\Gamma^d_A(B)_\ltdeg{m\alpha}.$$

We will now slightly generalize this in the following proposition.
Recall from paragraph~\pref{X:rho^e_d} that for $e\geq d$ we have a
\emph{surjection}
$\map{\rho^e_d}{\Gamma^e_A(B)}{\Gamma^d_A(B)}$ such that if
$\nu\in\N^{(\monpos)}$
$$\rho^e_d\bigl(\msb{\nu}\times\gammaid{e-|\nu|}\bigr)=
\begin{cases}
\msb{\nu}\times\gammaid{d-|\nu|} & \text{if $|\nu|\leq d$}\\
0 & \text{if $|\nu|> d$}.
\end{cases}$$
In particular, basis elements of length $>d$ are mapped to zero.
\end{xpar}

\begin{proposition}[Generalized Newton relations]
\label{P:generalized_classical_NR}
Let $x^\alpha \in B$ and $k,m$ be positive integers such that $k\leq d$. Then
$$\gamma^k(x^{m\alpha})\times \gammaid{d-k}-
(-1)^{km-k} m\gamma^{km}(x^\alpha)\times\gammaid{d-km}
\in\Gamma^d_A(B)_\ltdeg{km\alpha}$$
if $km\leq d$ and
$$\gamma^k(x^{m\alpha})\times\gammaid{d-k}\in\Gamma^d_A(B)_\ltdeg{km\alpha}$$
if $km>d$.
\end{proposition}
\begin{proof}
Using the homomorphism $\rho^e_d$ defined in paragraph~\pref{X:rho^e_d} with
$e\geq km$ we can assume that $km\leq d$. Further using the homomorphism
$\Gamma^d_{\Z}(\Z[t])\to\Gamma^d_A(A[t])\to\Gamma^d_A(B)$ where the second map
is induced by $A[t]\to B$, mapping $t$ to $x^\alpha$, it is enough to show that
$$\gamma^k(t^m)\times \gammaid{d-k}-
(-1)^{km-k} m\gamma^{km}(t)\times\gammaid{d-km}
\in\Gamma^d_{\Z}(\Z[t])_\ltdeg{km}.$$
Since $\Gamma^d_{\Z}(\Z[t])$
is a polynomial ring in $e_1(t),\dots,e_d(t)$, where
$e_i(t)=\gamma^i(t)\times\gammaid{d-i}$, we can write $e_k(t^m)$
\emph{uniquely} as a polynomial in $e_1(t),\dots,e_{km}(t)$. Clearly, all
terms of this polynomial will be in $\Gamma^d_{\Z}(\Z[t])_\ltdeg{km}$ except
the term $e_{km}(t)$. To determine the coefficient of $e_{km}(t)$ we tensor
with $\Q$.
The classical Newton identities, equation~\eqref{E:Newton_identities}, show
that
\begin{align*}
e_1\bigl((t^{m})^k\bigr)+(-1)^{k}ke_k(t^m)
  &\in \Gamma^d_{\Q}(\Q[t^m])_\ltdeg{km}\\
e_1(t^{km})+(-1)^{km}km e_{km}(t) &\in \Gamma^d_{\Q}(\Q[t])_\ltdeg{km}
\end{align*}
and thus $e_k(t^m)=(-1)^{km-k}m e_{km}(t)+\Gamma^d_{\Z}(\Z[t])_\ltdeg{km}$.
\end{proof}

\begin{corollary}\label{C:gen_of_Gamma_in_one_degree}
Let $x^{k\alpha}\in B$ be a monomial with $\gcd(\alpha)$ invertible in $A$.
Then the subalgebra
$A\bigl[\Gamma^d_A(B)_{\beta,\beta\leq k\alpha}\bigr] \subseteq \Gamma^d_A(B)$
is generated by $\Gamma^d_A(B)_\ltdeg{k\alpha}$ together with, if
$k\leq d$, the element~$\gamma^k(x^\alpha)\times\gammaid{d-k}$.
\end{corollary}
\begin{proof}
Follows immediately from Propositions~\pref{P:Newton_generators}
and~\pref{P:generalized_classical_NR}.
\end{proof}

\begin{corollary}\label{C:Newton_generators2}
$\Gamma^d_A(B)$ is generated as an $A$-algebra by the elements
$$\bigl(\gamma^{k}(x^\alpha)\times
\gammaid{d-k}\bigr)_{(k,x^\alpha)\in\calC}$$
where $\calC\subseteq\{1,2,\dots,d\}\times\monpos$ is one of the
collections
\begin{align*}
\calC_1 &= \{(k,x^\alpha)\;:\; \text{$\gcd(\alpha)\cdot 1_A$ invertible,
$k$ product of primes in $\primes(A)$}\}\\*[-1 mm]
\calC_2 &= \{(k,x^\alpha)\;:\; \gcd(\alpha)=1\}.
\end{align*}
\end{corollary}
\begin{proof}
If we let $\calC=\{1,2,\dots,d\}\times\monpos$ be the full set
of indices $(k,x^\alpha)$ then the corresponding set of elements
$\bigl\{\gamma^k(x^\alpha)\times\gammaid{d-k}\bigr\}$ is a generating set of
$\Gamma^d_A(B)$
by Proposition~\pref{P:Newton_generators}. That $\calC_1$ and $\calC_2$ also
give generating sets of $\Gamma^d_A(B)$ follows from
Corollary~\pref{C:gen_of_Gamma_in_one_degree}.
\end{proof}

\begin{remark}
Both generating sets of Corollary~\pref{C:Newton_generators2} have exactly
one generator of each multidegree in $\N^r\setminus 0$.
If $A=\Z$ then the two collections $\calC_1$ and $\calC_2$ coincide.
In~\cite[pf. of Thm. 1]{vaccarino_gen_TS} Vaccarino gives a proof of
Corollary~\pref{C:Newton_generators2} with the second collection using a
slightly different version of Proposition~\pref{P:Newton_generators}.
As it is sometimes convenient to also have the first collection we will use
either collection. Besides, all proofs work equally well with both collections.
\end{remark}

\begin{theorem}[{\cite[Thm. 1]{vaccarino_gen_TS}}]\label{T:Vaccarino_result}
The ring of multisymmetric functions $\Gamma^d_A(A[x_1,x_2,\dots,x_r])\iso
\TS^d_A(A[x_1,x_2,\dots,x_r])$
is generated as an $A$-algebra by $\gamma^k(x^\alpha)\times\gammaid{d-k}$ where
$\left(k,x^\alpha\right)\in\{1,2,\dots,d\}\times\monpos$ is such that
$\gcd(\alpha)=1$ and the total degree
$|\mdeg\bigl(\gamma^k(x^\alpha)\times\gammaid{d-k}\bigr)|=k|\alpha|$ is less
than or equal to $\max\bigl(d,r(d-1)\bigr)$.
\end{theorem}
\begin{proof}
We repeat the proof in~\cite{vaccarino_gen_TS}.
A result of Fleischmann~\cite[Thm. 4.6]{fleischmann}, cf.
Corollaries~\pref{C:multidegree_bound}
and~\pref{C:non-sharp_total_degree_bound},
shows that $\Gamma^d_A(B)$ is generated by the elements of total degree
$\leq{\max\bigl(d,r(d-1)\bigr)}$. The theorem then follows from
Corollary~\pref{C:Newton_generators2} using the second collection.
\end{proof}

\begin{remark}
As mentioned in the introduction, the generating set of
Theorem~\pref{T:Vaccarino_result} is not minimal. We will proceed to find a
minimal subset of either the first or second collection of
Corollary~\pref{C:Newton_generators2} which generates $\Gamma^d_A(B)$. Such
a minimal subset is then \emph{unique}. In fact, as the generators of
both collections are homogeneous and each multidegree occurs exactly once, any
minimal subset of $\calC$ which generates $\Gamma^d_A(B)$ consists of the
elements $\gamma^k(x^\alpha)\times\gammaid{d-k}$ such that
$\gamma^k(x^\alpha)\times\gammaid{d-k}\notin \Gamma^d_A(B)_\ltdeg{k\alpha}$.
\end{remark}

\begin{remark}
To determine if $\gamma^k(x^\alpha)\times\gammaid{d-k}\notin
\Gamma^d_A(B)_\ltdeg{k\alpha}$ it will be useful to lift this relation to
a relation in $\Gamma^n_A(B)$ where $n$ is an integer such that
$n\geq k|\alpha|$, cf. the proof of Theorem~\pref{T:Gamma_generators}.
In several of the following results involving $x^\alpha$ we will therefore use
$\Gamma^n_A(B)$ instead of $\Gamma^d_A(B)$ with an $n$ such that
$n\geq k|\alpha|$.
\end{remark}

\begin{remark}\label{R:relations_between_elem_and_Newton}
Let $x^\alpha\in\monpos$ and choose $n\in\N$ such that $|\alpha|\leq n$. The
multisymmetric Newton identities, Proposition~\pref{P:multi_symm_newton_ids},
show that in $\Gamma^n_A(B)$
\begin{equation}\label{ER:rbeN}
((\alpha))p_\alpha(x)+(-1)^{|\alpha|}|\alpha| e_\alpha(x)
\in \Gamma^n_A(B)_\ltdeg{\alpha}.
\end{equation}
If $|\alpha|$ divides $((\alpha))$, the image of $|\alpha|$ in $A$ is not a
zero divisor, and the LHS of
relation~\eqref{ER:rbeN} belongs to $|\alpha|\Gamma^n_A(B)_\ltdeg{\alpha}$,
then we obtain the relation
$$\frac{\bigl(|\alpha|-1\bigr)!}{\alpha_1!\dots \alpha_r!} p_\alpha(x)
+(-1)^{|\alpha|}e_\alpha(x)\in\Gamma^n_A(B)_\ltdeg{\alpha}.$$
For arbitrary $\alpha$ and $A$ this relation is not true, but we will show that
there exist similar relations between $\gamma^k(x^\alpha)\times\gammaid{d-k}$
and $e_{k\alpha}(x)$. We will first show that if such a relation exists then
it is unique.
\end{remark}

\begin{proposition}\label{P:uniqueness_of_contraction_coeff}
Let $x^{k\alpha}$ be a monomial in $B$ such that $\gcd(\alpha)$ is invertible
in $A$ and $n\geq k|\alpha|$. Let $a\in A$ be such that
$$a\gamma^k(x^\alpha)\times\gammaid{n-k}\in \Gamma^n_A(B)_\ltdeg{k\alpha}.$$
Then $a=0$.
\end{proposition}
\begin{proof}
Let $A'\surj A$ be any lifting of $A$ to a ring of characteristic $0$, e.g.
$A'=\Z[ \left(T_{a}\right)_{a\in A} ]$, and let $I=\ker(A'\surj A)$ and
$B'=A'[x_1,x_2,\dots,x_r]$. We have an induced surjection
$\Gamma^n_{A'}(B')\surj \Gamma^n_{A'}(B')\otimes_{A'} A\iso
\Gamma^n_A(B)$. Let $a'\in A'$ be a lifting of $a$. Then
$$a'\gamma^k(x^\alpha)\times\gammaid{n-k}+
\sum_{\begin{subarray}{l}\nu \in\N^{(\mon)}\\|\nu| =n\end{subarray}} i_\nu \msb{\nu}
\in \Gamma^n_{A'}(B')_\ltdeg{k\alpha}$$
for some $i_\nu\in I$. Since $\Gamma^n_{A'}(B')_\ltdeg{ka}$ is graded by
multidegree, taking the part with multidegree $k\alpha$ we obtain by
Corollary~\pref{C:gen_of_Gamma_in_one_degree}
\begin{equation}\label{E:uniqueness-eq1}
(a'+i)\gamma^k(x^\alpha)\times\gammaid{n-k}
\in\Gamma^n_{A'}(B')_\ltdeg{k\alpha}
\end{equation}
with $i\in I$. The homomorphism
${B'=A'[x_1,x_2,\dots,x_r]\surj A'[t]}$, taking $x_j$ to $t$ for every
$1\leq j\leq r$, induces a homomorphism
$\Gamma^n_{A'}(B')\surj\Gamma^n_{A'}(A'[t])\iso
A'[e_1(t),e_2(t),\dots,e_n(t)]$ which applied to
equation~\eqref{E:uniqueness-eq1} gives
$$(a'+i)\gamma^k( t^{|\alpha|} )\times\gammaid{n-k}
\in \Gamma^n_{A'}(A'[t])_\ltdeg{k|\alpha|}.$$
Thus by the generalized Newton relations of
Proposition~\pref{P:generalized_classical_NR} it follows that
$$(a'+i)|\alpha|e_{k|\alpha|}(t)\in\Gamma^n_{A'}(A'[t])_\ltdeg{k|\alpha|}.$$
As $\Gamma^n_{A'}(A'[t])_\ltdeg{k|\alpha|}=
A'[e_1(t),e_2(t),\dots,e_{k|\alpha|-1}(t)]$ and $\Gamma^n_{A'}(A'[t])$ is
a polynomial ring we see that $(a'+i)|\alpha|=0$. But the integer $|\alpha|$
is not a zero divisor in $A'$ by construction. Hence $a'+i=0$ and a
fortiori $a=0$.
\end{proof}

As an immediate corollary of
Proposition~\pref{P:uniqueness_of_contraction_coeff} we see that the generators
of total degree $\leq d$ in any of the collections of
Corollary~\pref{C:Newton_generators2} are contained in the minimal generating
subset:

\begin{corollary}\label{C:low-deg-necessary}
Let $x^{k\alpha}\in\monpos$ then
$\gamma^k(x^\alpha)\times\gammaid{d-k}\notin\Gamma^d_A(B)_\ltdeg{k\alpha}$ if
$k|\alpha|\leq d$.
\end{corollary}

We now establish relations of the kind mentioned
in Remark~\pref{R:relations_between_elem_and_Newton}.

\begin{proposition}\label{P:identities_between_elem_and_Newton}
Let $f_1,f_2,\dots,f_s\in\monpos$ be non-trivial monomials in $B$ and
$\ell_1,\ell_2,\dots,\ell_s$
positive integers such that $f_1^{\ell_1}\dots f_s^{\ell_s}=x^{k\alpha}$ with
$g=\gcd(\alpha)$ invertible in $A$ and let $n\geq k|\alpha|$.
Then $m=((\ell))\frac{kg}{|\ell|}\in \Z$ and $a=m/g\in A$ is the unique
element in $A$ such that
\begin{equation}\label{E:id_elem_Newton-eq1}
\gamma^\ell(f) \times \gammaid{n-|\ell|}
- (-1)^{|\ell|-k} a\gamma^k(x^\alpha)\times\gammaid{n-k}
\in \Gamma^n_A(B)_\ltdeg{k\alpha}.
\end{equation}
\end{proposition}
\begin{proof}
The existence and uniqueness of $a$ follow from
Corollary~\pref{C:gen_of_Gamma_in_one_degree}
and Proposition~\pref{P:uniqueness_of_contraction_coeff} respectively.
Proposition~\pref{P:generalized_classical_NR} shows that
$$\gamma^\ell(f) \times \gammaid{n-|\ell|}
- (-1)^{|\ell|-k'} ag\gamma^{k'}(x^{\alpha'})\times\gammaid{n-k'}
\in \Gamma^n_A(B)_\ltdeg{k'\alpha'}$$
with $k'=kg$ and $\alpha'=\alpha/g$. Replacing $\alpha$ with $\alpha'$ and
$k$ with $k'$ we can thus assume that $g=1$.

To determine the value of $a$ it is now enough to consider the case when
$A=\Z$ as $\Gamma^n_A(B)={\Gamma^n_{\Z}(\Z[x_1,x_2,\dots,x_r])\otimes_{\Z} A}$
and $a$ is unique. This also shows that $a$ is the image of an integer in $\Z$.
Multiplying both sides of equation~\eqref{E:id_elem_Newton-eq1} with
$\ell!=\ell_1!\ell_2!\dots\ell_s!$ we obtain
$$\left(\bigtimes_{i=1}^s f_i^{\times \ell_i}\right) \times \gammaid{n-|\ell|}
- (-1)^{|\ell|-k} a' \gamma^k(x^\alpha)\times\gammaid{n-k}
\in \Gamma^n_A(B)_\ltdeg{k\alpha}$$
with $a'=\ell!a$. As $B$ is a free $A$-module so is $\Gamma^d_A(B)$
by paragraph~\pref{X:Gamma_basis}. Thus $\ell!$ is not a zero divisor in
$\Gamma^n_A(B)$ and it is enough to verify
that $a'=\bigl(|\ell|-1\bigr)!k$. Replacing $\ell$ and $f$ with
$\ell'=1^{|\ell|}$ and $f'=(f_1,f_1,f_1,\dots,f_s,f_s)$ we can thus assume that
$\ell=1^{|\ell|}$. Further, using
Proposition~\pref{P:generalized_classical_NR} it is enough to show that
$$\left(\bigtimes_{i=1}^s f_i\right) \times \gammaid{n-s}
- (-1)^{s-1} (s-1)! \gamma^1(x^{k\alpha})\times\gammaid{n-1}
\in \Gamma^n_A(B)_\ltdeg{k\alpha}.$$
This is obvious for $s=1$. For $s>1$ we have by induction on $s$ that
\begin{align*}
\left(\bigtimes_{i=1}^s f_i\right) \times \gammaid{n-s}
&= \bigl(f_1\times\dots\times f_{s-1}\times\gammaid{n-(s-1)}\bigr)
   \bigl(f_s\times\gammaid{n-1}\bigr)\\
&\quad- \sum_{i=1}^{s-1} f_1\times f_2\times\dots
    \times f_i f_s\times\dots\times f_{s-1}\times\gammaid{n-(s-1)}\\
&= -(s-1) (-1)^{s-2} (s-2)! \gamma^1(x^{k\alpha})\times\gammaid{n-1}
    +\Gamma^n_A(B)_\ltdeg{k\alpha}
\end{align*}
which completes the proof.
\end{proof}

\begin{corollary}\label{C:sharp_contraction_coefficients_<=d-case}
Let $x^{k\alpha}$ be a monomial such that $g=\gcd(\alpha)$ is invertible in
$A$ and let $n\geq k|\alpha|$. Then there exists a unique element
$a\in \Z\cdot 1_A\subseteq A$ such that
$$(a/g)\gamma^k(x^\alpha)\times\gammaid{n-k} - e_{k\alpha} \in
\Gamma^n_A(B)_\ltdeg{k\alpha}.$$
For every prime $p\in\primes(A)$ we have that
$$\ord_p(a)=\ord_p ((k\alpha))-\ord_p\bigl(|\alpha|\bigr).$$
\end{corollary}
\begin{proof}
From Proposition~\pref{P:identities_between_elem_and_Newton} it follows
that $a=(-1)^{k|\alpha|-k}((k\alpha))\frac{kg}{k|\alpha|}$ and thus that
$\ord_p(a)=\ord_p ((k\alpha))-\ord_p\bigl(|\alpha|\bigr)$.
\end{proof}

We are now able to completely characterize the cases in which the elementary
symmetric polynomials generate all elements of total degree $\leq d$ in
$\Gamma^d_A(B)$.

\begin{lemma}\label{L:elementary_gen_low_deg_Newton}
Let $x^{k\alpha}$ be a monomial in $B$ such that $\gcd(\alpha)$ is invertible
in $A$ and $k|\alpha|\leq d$. Then the following statements are equivalent:
\begin{enumerate}
\item $A\bigl[\Gamma^d_A(B)_{\beta,\beta\leq k\alpha}\bigr]
\subseteq \Gamma^d_A(B)$ is
generated by $\Gamma^d_A(B)_\ltdeg{k\alpha}$ and $e_{k\alpha}$.
\label{EN:L:low_deg_1}
\item $\gamma^k(x^\alpha)\times \gammaid{d-k}\in
\Gamma^d_A(B)_\ltdeg{k\alpha}+Ae_{k\alpha}$.
\label{EN:L:low_deg_2}
\item $\ord_p\;((k\alpha))-\ord_p\bigl(|\alpha|\bigr)=0$ for every
$p\in\primes(A)$ such that $p \leq d$.
\label{EN:L:low_deg_3}
\end{enumerate}
\end{lemma}
\begin{proof}
\enumref{EN:L:low_deg_1}$\iff$\enumref{EN:L:low_deg_2} by
Corollary~\pref{C:gen_of_Gamma_in_one_degree}.
\enumref{EN:L:low_deg_2}$\iff$\enumref{EN:L:low_deg_3} follows from
Corollary~\pref{C:sharp_contraction_coefficients_<=d-case}.
\end{proof}

\begin{proposition}\label{P:elementary_gen_low_deg_Newton}
The subalgebra $A\bigl[\Gamma^d_A(B)_{\beta,|\beta|\leq d}]\subseteq
\Gamma^d_A(B)$ is generated by the elementary multisymmetric polynomials
$(e_\alpha)_{|\alpha|\leq d}$
where $e_\alpha=\gamma^\alpha(x)\times\gammaid{d-|\alpha|}$, if and only
if one of the following conditions is satisfied
\begin{enumerate}
\item $r=1$.\label{EN:P:low_deg_1}
\item $r=2$ and $d=3$.\label{EN:P:low_deg_2}
\item $r=2$, $d=4$ and $3$ is invertible in $A$.\label{EN:P:low_deg_3}
\item $(d-1)!$ is invertible in $A$.\label{EN:P:low_deg_4}
\end{enumerate}
\end{proposition}
\begin{proof}
When $r=1$ it is well known that $\Gamma^d_A(B)$ is the polynomial ring
$A[e_1,e_2,\dots,e_d]$ which shows that \enumref{EN:P:low_deg_1} is sufficient.

By Lemma~\pref{L:elementary_gen_low_deg_Newton}
and induction on $|\alpha|$ it is enough to check that
$$\ord_p\;((k\alpha))-\ord_p\bigl(|\alpha|\bigr)=0$$
for every monomial $x^{k\alpha}$ and $p\in\primes(A)$ such that
$p\leq d$, $k|\alpha|\leq d$ and $\gcd(\alpha)$ is invertible in $A$.

If $d$ itself is a prime in $\primes(A)$ then this prime has not to be
checked. In fact, if $d\divides k\alpha$ then $d=k$ and
$\alpha=\basis{i}$ for some $i$ giving $\ord_d\;((k\basis{i}))-\ord_d(1)=0$.
If $d\notdivide k\alpha$ then either $k|\alpha|<d$ which gives
$\ord_d\;((k\alpha))-\ord_d\bigl(|\alpha|\bigr)=0-0=0$
or $k=1$ and $|\alpha|=d$ which gives
$\ord_d\;((\alpha))-\ord_d\bigl(|\alpha|\bigr)=1-1=0$.
It is thus sufficient that $(d-1)!$ is invertible which is condition
\enumref{EN:P:low_deg_4}.

If a prime $2<p<d$ is not invertible and $r\geq 2$ then
$$\gamma^{1}(x_1^{p-1}x_2^2)\times \gammaid{d-1}=
\gamma^1(x^\alpha)\times \gammaid{d-1}$$
with $\alpha=(p-1)\basis{1}+2\basis{2}$ is an element which is not generated by
elementary multisymmetric polynomials. In fact
$\ord_p\;((\alpha))-\ord_p\bigl(|\alpha|\bigr)=1-0=1$.

It is thus necessary that every prime $p<d$, except possibly $2$, is
invertible. Therefore we need only consider the case where $2$ is not
invertible and $d\geq 3$. If $r\geq 3$ then
$$\gamma^1(x_1x_2x_3)\times \gammaid{d-1}$$
is not generated by the $e_\alpha$:s as
$\ord_2\;((1,1,1))-\ord_2(1+1+1)=1-0=1$. If $r=2$ and $d\geq 5$ then
$$\gamma^1(x_1^3x_2^2)\times \gammaid{d-1}$$
is not generated by the $e_\alpha$:s since
$\ord_2\;((3,2))-\ord_2(3+2)=1-0=1$. Finally if $r=2$ and $d\leq 4$ then
$\ord_2\;((\alpha))-\ord_2\bigl(|\alpha|\bigr)=0$ for all
$\alpha\in \{(1,1),(2,1),(3,1)\}$ and this completes the proof of the
proposition.
\end{proof}

%%%%%%%%%%%%%%%%%%%%%%%%

We will now show the main theorem of this section. It gives a \emph{minimal}
generator set for $\Gamma^d_A(B)$ where $A$ is any ring and
improves~\cite[Thm. 1]{vaccarino_gen_TS} also when $A=\F_p$ and $d=p^s$. A
sharp bound on the total degree for any $A$ is given in
Corollary~\pref{C:sharp_total_degree_bound}.

\begin{theorem}\label{T:Gamma_generators}
Let $\calC$ be one of the two collections of
Corollary~\pref{C:Newton_generators2}.
Let $\widetilde{\calC}$ be the subset of $\calC$ such that
$(k,x^\alpha)\in\widetilde{\calC}$ if either $k|\alpha|\leq d$
or $\Qp{k\alpha}\leq\Qp{d}$ for some $p\in\primes(A)$.

The $A$-algebra $\Gamma^d_A(B)$ is then generated by
$\bigl(\gamma^k(x^\alpha)\times\gammaid{d-k}\bigr)_{(k,\alpha)
\in\widetilde{\calC}}$ and
this is a minimal set of generators.
\end{theorem}
\begin{proof*}
By Corollary~\pref{C:Newton_generators2}, the elements
$\gamma^k(x^\alpha)\times\gammaid{d-k}$ with
$(k,\alpha)\in\calC$ generate $\Gamma^d_A(B)$. As every multidegree occurs
exactly once in $\calC$ it is clear that we get a minimal set of
generators by taking those $\gamma^k(x^\alpha)\times\gammaid{d-k}$ which cannot
be written as sums of products of elements of strictly smaller multidegree,
i.e. $\gamma^k(x^\alpha)\times\gammaid{d-k}$ is in the minimal set if and only
if $\gamma^k(x^\alpha)\times\gammaid{d-k}\notin\Gamma^d_A(B)_\ltdeg{k\alpha}$.

If $k|\alpha|\leq d$ then Corollary~\pref{C:low-deg-necessary}
shows that $\gamma^k(x^\alpha)\times\gammaid{d-k}
\notin\Gamma^d_A(B)_\ltdeg{k\alpha}$.
If $k|\alpha|>d$ and $\gamma^k(x^\alpha)\times\gammaid{d-k}
\in\Gamma^d_A(B)_\ltdeg{k\alpha}$ then
we lift the corresponding relation in $\Gamma^d_A(B)$ to $\Gamma^n_A(B)$,
where $n=k|\alpha|$, using the homomorphism
$\surjmap{\rho^n_d}{\Gamma^n_A(B)}{\Gamma^d_A(B)}$ defined in
paragraph~\pref{X:rho^e_d} and obtain
\begin{equation}\label{E:GG}
\gamma^k(x^\alpha)\times\gammaid{n-k} = 
\sum_{\substack{\nu\in \N^{(\monpos)}\\d+1\leq |\nu|\leq n}} a_\nu \msb{\nu}
\times\gammaid{n-|\nu|}
+ \Gamma^n_A(B)_\ltdeg{k\alpha}
\end{equation}
for some $a_\nu\in A$. Conversely, if there exist $a_\nu\in A$ such that
relation~\eqref{E:GG} holds then
$\gamma^k(x^\alpha)\times\gammaid{d-k}\in\Gamma^d_A(B)_\ltdeg{k\alpha}$.
The theorem now follows from the following lemma:
\end{proof*}

\begin{lemma}\label{L:sharp_contraction_coefficients}
Let $x^{k\alpha}$ be a monomial such that $\gcd(\alpha)$ is invertible in
$A$. Let $d$ and $n$ be positive integers such that $d<k|\alpha|\leq n$. Then
there exists a relation in $\Gamma^{n}_A(B)$ of the form
\begin{equation}\label{E:scc-1}
\gamma^k(x^\alpha)\times\gammaid{n-k} = 
\sum_{\substack{\nu\in\N^{(\monpos)}\\d<|\nu|\leq n}} a_\nu \msb{\nu}
\times\gammaid{n-|\nu|}
+ \Gamma^n_A(B)_\ltdeg{k\alpha}
\end{equation}
where $a_\nu\in A$ are almost all zero, if and only if $\Qp{k\alpha}>\Qp{d}$
for every prime $p\in\primes(A)$.
\end{lemma}
\begin{proof}
We can assume that every term $\msb{\nu}\times\gammaid{n-|\nu|}$ has
multidegree $k\alpha$. The sum is then over the set $\calS_{k\alpha,d}$
of Definition~\pref{D:calS}.
Proposition~\pref{P:identities_between_elem_and_Newton} gives that
$$\msb{\nu}\times\gammaid{n-|\nu|}
-c_\nu\gamma^k(x^\alpha)\times\gammaid{n-k}
\in\Gamma^n_A(B)_\ltdeg{k\alpha}$$
where $c_\nu=(-1)^{|\nu|-k}k((\nu))/|\nu|$. Thus
\begin{equation}\label{E:contracted_identity}
(1-\sum_{\nu\in\calS_{k\alpha,d}} a_\nu c_\nu)
\bigl(\gamma^k(x^\alpha)\times\gammaid{n-k}\bigr)
\in \Gamma^n_A(B)_\ltdeg{k\alpha}
\end{equation}
if and only if~\eqref{E:scc-1} holds. Moreover,
relation~\eqref{E:contracted_identity} is equivalent to
$1=\sum_{\nu} a_\nu c_\nu$ in $A$ by
Proposition~\pref{P:uniqueness_of_contraction_coeff}. If
$1=\sum_{\nu} a_\nu c_\nu$ then for every $p\in\primes(A)$ there
exists a $\nu\in\calS_{k\alpha,d}$ such that $\ord_p(c_\nu)=0$.
Conversely, if this is the case we can choose $a_\nu\in\Z$ such
that $\sum_{\nu} a_\nu c_\nu$ is invertible in $A$.
By Main Lemma~\pref{ML:min_order} the existence of a $\nu$ such that
$\ord_p(c_\nu)=0$ is equivalent to $\Qp{k\alpha}>\Qp{d}$. This concludes the
proof of the lemma.
\end{proof}

\end{section}

%%%%%%%%%%%%%%%%%%%%%%%%%%%%%%%%%%%%%%%%%%%%%%%%%%%%%%%%%%%%%%%%%%%%%%%%%

\begin{section}{Remarks and applications}\label{S:applications}

\begin{remark}\label{R:Gamma_generators_only_p<=d}
In Theorem~\pref{T:Gamma_generators} it is enough to consider non-invertible
primes $\leq d$ since if $p>d$ then $\Qp{k\alpha}>\Qp{d}=d$ for any
$k|\alpha|>d$.
\end{remark}

\begin{remark}
Let us extend the definition of $\Qp{n}$ to include $p=\infty$ with
$\Qpol{\infty}{n}=n\in \Z[t]$.
We can then replace the condition in Theorem~\pref{T:Gamma_generators}
with: $(k,\alpha)\in\widetilde{\calC}$ if and only if
$\Qp{k\alpha}\leq \Qp{d}$ for some $p\in\primes(A)\cup\{\infty\}$.
\end{remark}

\begin{remark}
We note that it immediately follows from Theorem~\pref{T:Gamma_generators} that
if $\gamma^k(x^\alpha)\times\gammaid{d-k}$ is in a minimal set of generators
then so is $\gamma^{k'}(x^{\alpha'})\times\gammaid{d-k'}$ for every
$k'\alpha'<k\alpha$.
\end{remark}

% cite[Prop. 2]{richman}
\begin{corollary}[{\cite{junker_93,noether_endlichkeitssatz,weyl_invariants,
nagata_chownorm,neeman,richman}}]\label{C:generators_in_char_0}
If $d!$ is invertible in $A$ then $\Gamma^d_A(A[x_1,\dots,x_r])$ is minimally
generated as an $A$-algebra by either the elementary multisymmetric polynomials
or the multisymmetric power sums of total degree $\leq d$.
\end{corollary}
\begin{proof}
From Theorem~\pref{T:Gamma_generators} and
Remark~\pref{R:Gamma_generators_only_p<=d} we deduce that the elements of
total degree $\leq d$ generates $\Gamma^d_A(A[x_1,\dots,x_r])$. The statement
then follows from Corollary~\pref{C:elem_equiv_power_sums_low_deg}.
\end{proof}

\begin{corollary}\label{C:generators_in_char_p}
Let $A$ be of equal or mixed characteristic $p$, i.e. $p$ is the only
non-invertible prime in $A$.
Then $\Gamma^d_A(A[x_1,\dots,x_r])$ is minimally generated as an $A$-algebra
by the elements $\gamma^{p^s}(x^\alpha)\times\gammaid{d-p^s}$ with
$s\in\N$ and $\alpha\in\N^r\setminus 0$ such that $p\notdivide \alpha$ and
$\Qp{p^s\alpha}\leq \Qp{d}$
or equivalently $\Qp{\alpha}\leq \Qpp{\floor{d/p^s}}$.
\end{corollary}
\begin{proof}
Follows immediately from Theorem~\pref{T:Gamma_generators} using the
first collection of Corollary~\pref{C:Newton_generators2}.
\end{proof}

\begin{corollary}[{\cite[Thm. 4.6]{fleischmann}}]\label{C:multidegree_bound}
Let $A$ be an arbitrary ring. Then
$\Gamma^d_A(A[x_1,\dots,x_r])$ is generated as an $A$-algebra by
$\gamma^d(x_1),\gamma^d(x_2),
\dots,\gamma^d(x_r)$ and the elements $\gamma^k(x^\alpha)\times\gammaid{d-k}$
with $k\alpha\leq(d-1,d-1,\dots,d-1)$. Further, there
is no smaller multidegree bound and if $d=p^s$ for some prime $p$ not
invertible in $A$, then $\Gamma^d_A(A[x_1,\dots,x_r])$ is not generated
by elements of strictly smaller multidegree.
\end{corollary}
\begin{proof}
If $k\alpha_i\geq d$ and $|k\alpha|>d$ then $\Qp{k\alpha}>\Qp{d}$ for any
prime $p$ which shows that $\gamma^k(x^\alpha)\times\gammaid{d-k}$
is not among the minimal generators of Theorem~\pref{T:Gamma_generators}. On
the other hand ${(d-1)\basis{i}}={(0,\dots,0,{d-1},0,\dots,0)\in\N^r}$ is the
multidegree of a minimal generator and it follows that there is no smaller
multidegree bound. If $d=p^s$ then $\Qpp{(d-1,d-1,\dots,d-1)}<\Qp{d}$
which shows that there is an element of every multidegree
$\leq (d-1,d-1,\dots,d-1)$ in any generating set.
\end{proof}

From Corollary~\pref{C:multidegree_bound} we immediately obtain:

\begin{corollary}[{\cite[Thm. 4.6, 4.7]{fleischmann}}]
\label{C:non-sharp_total_degree_bound}
If $A$ is an arbitrary ring then $\Gamma^d_A(A[x_1,\dots,x_r])$ is generated as
an $A$-algebra by elements $\gamma^k(x^\alpha)\times\gammaid{d-k}$ of total
degree $k|\alpha|\leq {\max\bigl(d,r(d-1)\bigr)}$. Further, this total degree
bound is sharp if $d=p^s$ for some prime $p\in\primes(A)$.
\end{corollary}

A more careful examination of the conditions in the theorem gives a sharp
total degree bound on the generators of $\Gamma^d_A(B)$:

\begin{corollary}\label{C:sharp_total_degree_bound}
Let $d$ be an integer. For every prime $p$ we let $1\leq a_p\leq p-1$ and
$b_p\in\N$ be the unique integers such that $d=a_p p^{b_p}+c_p$ for some
$0\leq c_p<p^{b_p}$. For any ring $A$ the $A$-algebra
$\Gamma^d_A(A[x_1,\dots,x_r])$ is then
minimally generated by elements of total degree at most
$$\max\left\{d,\max_{p\in\primes(A)}\left( (a_p+r-1)p^{b_p}-r\right)\right\}$$
and every generating set contains an element \emph{attaining this bound}.
\end{corollary}
\begin{proof}
If $r=1$ then the bound becomes $d$ and is sharp as
$\Gamma^d_A(A[x])=A[e_1,e_2,\dots,e_d]$ so we will assume that $r\geq 2$.
Let $p$ be a prime not invertible in $A$ and
$\gamma^k(x^\alpha)\times\gammaid{d-k}\in\Gamma^d_A(B)$ an element of total
degree $k|\alpha|>\max\{d,(a-1)p^b+r(p^b-1)\}$ where $a=a_p$ and $b=b_p$.

If $\beta\in\N^r$ is such that $|\beta|>(l-1)p^m+r(p^m-1)$ for some integer
$1\leq l\leq p-1$ then there exists $\beta'\leq\beta$ such that
$\Qp{\beta'}=l t^m$ and we have that
\begin{align*}
|\beta-\beta'|>(r-1)p^m-r
&=\bigl(r(p-1)-p\bigr)p^{m-1}+r(p^{m-1}-1)\\
&\geq (p-2)p^{m-1}+r(p^{m-1}-1)
\end{align*}
as $r\geq 2$.

We can thus find $\alpha_b,\alpha_{b-1},\dots,\alpha_0$ such
that $\Qp{\alpha_b}=at^b$, $\Qp{\alpha_m}=(p-1)t^m$ for $m<b$ and
$k\alpha\geq \alpha_b+\alpha_{b-1}+\dots+\alpha_0$. This shows that
$$\Qp{k\alpha}\geq at^b+(p-1)t^{b-1}+(p-1)t^{b-2}+\dots+(p-1)\geq \Qp{d}$$
and as $k|\alpha|>d$ we have that $\Qp{k\alpha}\neq \Qp{d}$.
By Theorem~\pref{T:Gamma_generators} this implies that
$\gamma^k(x^\alpha)\times\gammaid{d-k}$ is generated by elements of lower
degree.

To show that the bound is attained, consider the element
$\gamma^k(x^\alpha)\times\gammaid{d-k}$ with
$$k\alpha=(ap^b-1,p^b-1,\dots,p^b-1), \quad \text{$\gcd(\alpha)$ invertible}$$
which is not generated by elements of lower degree since
\begin{align*}
\Qp{k\alpha}
&=\Qpp{(ap^b-1,p^b-1,\dots,p^b-1)}\\
&=(a-1)t^b+r(p-1)t^{b-1}+r(p-1)t^{b-2}+\dots+r(p-1)\\
&<\Qp{d}.
\end{align*}
\end{proof}

\begin{remark}
The inequality
$$(a_p+r-1)p^{b_p}-r\leq r\left(a_p p^{b_p}-1\right) \leq r(d-1)$$
with equality if and only if $d=p^{b_p}$, or $r=1$ and $d=a_p p^{b_p}$,
together with Corollary~\pref{C:sharp_total_degree_bound} gives another proof
of Corollary~\pref{C:non-sharp_total_degree_bound}. Further we see that the
total degree bound $\max\bigl(d,r(d-1)\bigr)$ is sharp if and only if
$r(d-1)\leq d$ or $d=p^s$, that is if and only if one of the following
conditions is satisfied
\begin{enumerate}
\item $r=1$.
\item $r=2$ and $d=2$.
\item $d=p^s$ with $p\in\primes(A)$.
\end{enumerate}
\end{remark}

\begin{corollary}[{\cite[Thm. 1]{briand_elem_multi_gens}}]
\label{C:elementary_gen}
$\Gamma^d_A(A[x_1,\dots,x_r])$ is generated as an $A$-algebra by elementary
multisymmetric polynomials if and only if one of the following conditions is
satisfied
\begin{enumerate}
\item $r=1$.\label{EN:C:elementary_gen_1}
\item $d!$ is invertible in $A$.\label{EN:C:elementary_gen_2}
\item $r=2$ and $d=2$.\label{EN:C:elementary_gen_3}
\item $r=2$, $d=3$ and $3$ is invertible in $A$.\label{EN:C:elementary_gen_4}
\end{enumerate}
\end{corollary}
\begin{proof}
If $r=1$ then $\Gamma^d_A(A[x_1,\dots,x_r])$ is the polynomial ring in the
elementary polynomials so \enumref{EN:C:elementary_gen_1} is sufficient and we
can assume that $r\geq 2$.
It then follows from Proposition~\pref{P:elementary_gen_low_deg_Newton} that
every prime such that $2<p<d$ is invertible in $A$. If $d>2$ is a prime
then $x_1^{d-1}x_2^2\times \gammaid{d-1}$ is a sum of products of elements
of total degree $\leq d$ if and only if $d$ is invertible by
Theorem~\pref{T:Gamma_generators} since $\Qpolp{d}{(d-1,2)}<\Qpol{d}{d}$.
Thus it is necessary that every prime such that $2<p\leq d$ is invertible in
$A$. On the other hand condition \enumref{EN:C:elementary_gen_2} is sufficient
by Corollary~\pref{C:generators_in_char_0}.

This leaves the case when $2$ is not invertible in $A$ but every
odd prime $\leq d$ is. By Proposition~\pref{P:elementary_gen_low_deg_Newton} we
can then assume that $r\geq 2$ and $d=2$ or $r=2$ and $d\leq 4$.
If $d=2$ and $r\geq 3$ then $\gamma^1(x_1 x_2 x_3)\times 1$ is not
generated by elements of lower degree as $\Qpolp{2}{(1,1,1)}<\Qpol{2}{3}$.
If $d=4$ then $\gamma^1(x_1^3 x_2^2)\times \gammaid{3}$ is not
generated by elements of lower degree as $\Qpolp{2}{(3,2)}<\Qpol{2}{4}$.
In the remaining cases, $r=2$ and $d=2$ or $d=3$, it is easily seen that
$\Qpol{2}{k\alpha}<\Qpol{2}{d}$ implies that $k|\alpha|\leq d$
and we can conclude with Proposition~\pref{P:elementary_gen_low_deg_Newton}.
\end{proof}

\begin{remark}\label{R:Sym-Chow}
Let $k$ be an algebraically closed field. It can be shown that the Chow scheme
$\Chow_{0,d}(\A{r}_k\inj\PR{r}_k)$, parameterizing $0$-cycles of degree $d$
in $\A{r}_k$, is isomorphic to $\Spec(C)$ where $C$ is the subring of
$\Gamma^d_k(k[x_1,x_2,\dots,x_r])\iso\TS^d_k(k[x_1,x_2,\dots,x_r])$ generated
by the elementary multisymmetric polynomials. This gives a morphism
$$\Sym^d(\A{r}_k):=\Spec\bigl(\TS^d_k(k[x_1,x_2,\dots,x_r])\bigr)\to
\Chow_{0,d}(\A{r}_k\inj\PR{r}_k)$$
which is an isomorphism exactly in the cases listed in
Corollary~\pref{C:elementary_gen}.

In general it is always possible to find a projective embedding
$\A{r}\inj\PR{N}$ such that
$$\Sym^d(\A{r}_k)\to\Chow_{0,d}(\A{r}_k\inj\PR{N}_k)$$
is an isomorphism. A bound on the degree of the generators of the ring
$\TS^d_k(k[x_1,x_2,\dots,x_r])$ such as
Corollary~\pref{C:multidegree_bound} gives an \emph{effective}
answer to the embedding needed to obtain such an isomorphism.

These issues are thoroughly discussed in~\cite{rydh_gammasymchow_inprep}.
\end{remark}

\begin{remark}
The results of \S\S\ref{S:gen_of_Gamma}-\ref{S:applications}
immediately generalize to the case where $B=A[(x_\alpha)_{\alpha\in\I}]$ is the
polynomial ring in an infinite number of variables. This is easily seen
considering statement by statement but as $B$ is the filtered direct limit of
finitely generated polynomial rings, it also follows directly from the fact
that $\Gamma^d_A(\cdot)$ commutes with filtered direct limits as shown in
paragraph~\pref{X:Gamma^d_filt_dir_lims}.
\end{remark}

\end{section}

%---------------------------------------------------------------

\bibliographystyle{dary}
\bibliography{GenForGamma}
\end{document}